\documentclass[a4paper,11pt]{article}

\usepackage{hyperref}

\usepackage{latexsym}
\usepackage{amsmath}
\usepackage{amsthm}
\usepackage{amscd}
\usepackage{amssymb}

\usepackage{url}

\usepackage{tikz}

\usepackage[mode=multiuser,status=draft,lang=english]{fixme}
\fxsetup{theme=colorsig}

\FXRegisterAuthor{K}{aK}{Keita}
\FXRegisterAuthor{S}{aS}{Silvia}

\usepackage[a4paper]{geometry}
\def\RCAo{\mathsf{RCA_0}}
\def\RCAs{\mathsf{RCA^*_0}}
\def\WKLo{\mathsf{WKL_0}}
\def\WKL{\mathsf{WKL_0}}
\def\ACAo{\mathsf{ACA_0}}
\def\RT{\mathsf{RT}}

\def\E{\exists}
\def\A{\forall}
\def\N{\mathbb{N}}

\def\Seq{\mathrm{Seq}}

\def\lh{\mathrm{lh}}

\def\id{\mathop{\mathrm{id}}\nolimits}

\def\rest{{\upharpoonright}}

\newcommand{\rng}{\mathrm{rng}}
\newcommand{\ran}{\mathrm{rng}}

\def\PH{\mathrm{PH}}

\def\ADS{\mathrm{ADS}}
\def\CAC{\mathrm{CAC}}

\def\Ii{\mathrm{I}\Sigma_1}

\def\BII{\mathrm{B}\Sigma^0_2}
\def\II{\mathrm{I}\Sigma^0_1}

\newcommand\LAR{\mathrm{LRG}}

\newcommand\rank{\mathrm{rank}}

\newcommand\WPH{\mathrm{WPH}}
\newcommand\PHs{\mathrm{PH}^{*}{}}
\newcommand\WPHs{\mathrm{WPH}^{*}{}}
\newcommand\HCT{\mathrm{HCT}}
\newcommand\HCTo{\mathrm{HCT}^{\omega}}
\newcommand\WRT{\mathrm{WRT}}

\newcommand\Ack{\mathcal{F}}

\newcommand\TTo{\mathrm{TT}^{*}_{\omega}}
\newcommand\Tot{\mathrm{Tot}}

\newcommand\TT{\mathrm{TT}}

\newcommand\HCToh{\mathrm{HCT}_{b}^{*}}
\newcommand\TToh{\mathrm{TT}_{b}^{*}}
\newcommand\HCToH{\mathrm{HCT}_{H}^{*}}
\newcommand\TToH{\mathrm{TT}_{H}^{*}}
\newcommand\TRT{\mathrm{TRT}^2_2}

\DeclareMathOperator{\tP}{\mathcal{P}}

\newcommand{\ap}[1]{\langle #1 \rangle}
\newcommand{\bp}[1]{\left\lbrace #1 \right\rbrace}

\newcommand\acc{\mathrm{Acc}}

\newcommand{\vgt}[1]{``#1''}

\newtheorem{thm}{Theorem}[section]
\newtheorem{theorem}[thm]{Theorem}

\newtheorem{claim}{Claim}[thm]
\newtheorem*{claim*}{Claim}
\newtheorem{prop}[thm]{Proposition}
\newtheorem{lem}[thm]{Lemma}
\newtheorem{cor}[thm]{Corollary}
\newtheorem{proposition}[thm]{Proposition}
\newtheorem{lemma}[thm]{Lemma}
\newtheorem{corollary}[thm]{Corollary}

\theoremstyle{definition}
\newtheorem{defi}{Definition}[section]
\newtheorem{definition}[defi]{Definition}

\newtheorem{remark}[thm]{Remark}

\newtheorem{question}[defi]{Question}
\newtheorem{example}[defi]{Example}

\title{Reverse mathematical bounds for the Termination Theorem}
\author{
Silvia Steila and
Keita Yokoyama%
\footnote{
e-mail: {\sf y-keita@jaist.ac.jp}
}
\footnote{This work is partially supported by JSPS Grant-in-Aid for Research Activity Start-up grant number 25887026.}
}
\date{\today}

\begin{document}
\maketitle

\begin{abstract}
In 2004 Podelski and Rybalchenko expressed the termination of transition-based programs as a property of well-founded relations. The classical proof by Podelski and Rybalchenko requires Ramsey's Theorem for pairs which is a purely classical result, therefore extracting bounds from the original proof is non-trivial task.

Our goal is to investigate the termination analysis from the point of view of Reverse Mathematics. By studying the strength of Podelski and Rybalchenko's Termination Theorem we can extract some information about termination bounds.
\end{abstract}

\section{Introduction}
In \cite{Podelski} Podelski and Rybalchenko characterized the termination of transition-based programs as a property of well-founded relations. Their result may be stated as follows: a binary relation $R$ is well-founded if and only if there exist a natural number $k$ and $k$-many well-founded relations whose union contains the transitive closure of $R$. The classical proof of Podelski and Rybalchenko's Termination Theorem (just Termination Theorem for short) requires Ramsey's Theorem for pairs. Although Ramsey's Theorem for pairs is a purely classical result, the Termination Theorem can be intuitionistically proved by using some intuitionistic version of Ramsey providing and providing to consider the intuitionistic notion of well-foundedness. In 2012 Vytiniotis, Coquand and Wahlstedt proved an intuitionistic version of the Termination Theorem by using the Almost-Full Theorem \cite{CoquandStop}, while in 2014 Stefano Berardi and the first author proved it by using the $H$-closure Theorem \cite{Hclosure}. The $H$-closure Theorem arose by the combinatorial fragment needed to prove the Termination Theorem (see Section \ref{Section: Ramsey}, \ref{Section: TTandHc}).

The goal of this paper is to study the $H$-closure Theorem and the Termination Theorem from the viewpoint of Reverse Mathematics, in order to extract bounds for termination. The first question is whether the $H$-closure Theorem and the Termination Theorem are equivalent over $\RCAo$ to Ramsey's Theorem for pairs. Due to our analysis we answer to  \cite[Open Problem 2]{Gasarch} posed by Gasarch: finding a natural example showing that the Termination Theorem requires the full Ramsey Theorem for pairs. In this paper we prove that such program cannot exists. We also answer negatively to \cite[Open Problem 3]{Gasarch} posed by Gasarch: is the Termination Theorem equivalent to Ramsey's Theorem for pairs? 

In \cite{Schmitz} Figueira et al. gave a deeper analysis of the Termination Theorem by using Dickson's Lemma\footnote{Dickson's Lemma states that $(\N^k, \leq)$ (where ${\leq}$ is the componentwise order) is a well-quasi order \cite{Kruskal, Milner}; i.e. every infinite sequence $\sigma$ of elements of $\N^k$ is such that there exists $n < m$ with $\sigma(n) \leq \sigma(m)$.}. In fact the Termination Theorem is a consequence of Dickson's Lemma by observing that any relation is well-founded if and only if it is embedded into a well-quasi-ordering. However this property of well-quasi-orderings is equivalent to $\ACAo$ over $\RCAo$ and therefore in order to analyse the strength of the Termination Theorem we need a different point of view. 

In Section \ref{Section: zoo} we prove that the Termination Theorem is equivalent over $\RCAo$ to a weak version of Ramsey's Theorem for pairs. As a corollary of this result we have that for any natural number $k$, $\CAC$ (the Chain-AntiChain principle) is stronger than the Termination Theorem for $k$ many relations, which is the statement: given a relation $R$, if there exist $k$-many well-founded relations $R_0, \dots, R_{k-1}$ such that $R_0 \cup \dots \cup R_{k-1} \supseteq R^+$ then $R$ is well-founded. Therefore we get answers to \cite[Open Problem 2, Open Problem 3]{Gasarch}.

These results can be used to characterize the programs proved to be terminating by the Termination Theorem: our goal is to extract a time bound for such a program by using reverse mathematics tools. Assume that $R$ is the binary transition relation of some program. We say that a function $f: S \to \N$ is a bound for the relation $R$ on $S$, if any $R$-decreasing sequence starting from an element $a\in S$ is shorter than $f(a)$. By using \cite{prov-rec-func, Paris-Kirby, CSY2012} it is known that the class provably recursive functions of $\WKLo+\CAC$ is exactly the same as the class of primitive recursive functions. Hence given any binary relation $R$ generated by a primitive recursive transition function, we conclude that if there exist $k$-many relations $R_1 \cup \dots \cup R_{k-1} \supseteq R^+$  with primitive recursive bounds, then the program has a primitive recursive bound. The proof is in Section \ref{Section: application}. 

In order to provide more precise termination bounds, in Section \ref{Section: boundedtermination} we study the reverse mathematical strength of some bounded versions of both the $H$-closure Theorem and of the Termination Theorem. Differently from the full case, in the restricted ones they turn out to be equivalent. Moreover we prove they are equivalent to a weaker version of the Paris Harrington Theorem \cite{PH}. 

A natural question which arise is: is there a correspondence between the complexity of a primitive recursive transition relation and the number of relations which compose the transition invariant? Thanks to our analysis and by using the relationship between Paris Harrington Theorem and the Fast-Growing Hierarchy in Section and \ref{Section: boundedbounds} we provide results in this directions.

Finally, in Section \ref{Section: iteratedversion}, we focus on the case of iterated applications of the Termination Theorem.

\section{Ramsey's Theorem in reverse mathematics}\label{Section: Ramsey}

Ramsey's Theorem, and in particular Ramsey's Theorem for pairs in two colors, is a central central argument of study in Reverse Mathematics. In this section we summarize some main facts about the strength of Ramsey's Theorem and of some of its corollaries we will use in this paper. Let $k$ be a natural number. Ramsey's Theorem for pairs in $k$ colors guarantees that for any coloring in $k$-many colors over the edges of a complete graph with countably many nodes, there exists an infinite homogeneous set; i.e. there exists an infinite subset $H$ of the vertices such that any two elements of $H$ are connected in the same color. 

A $n$-regular hypergraph is a set $X$ equipped with some set of subsets of $X$ having cardinality $n$. A $1$-regular hypergraph defines a subset of $X$ and a $2$-regular hypergraph defines a graph in which no element is related to itself. A $3$-regular hypergraph is a set with some triangular connection on its elements. Ramsey proved that his theorem holds also for $n$-regular hypergraphs, that is: given any set $X$ and any assignment of colors to all subsets of cardinality $n$, there is some infinite subset $H$ whose subset of cardinality $n$ all have the same color. Thus we can formally state Ramsey's Theorem for $n$-regular hypergraphs in $k$ colors as follows. Given a set $X$ we let $[X]^n$ denote the $n$-regular complete hypergraphs on $X$, i.e. $\bp{Y\subseteq X : |Y|=n}$.

\begin{definition} For given $k, n \in \N$, $k \geq 2$ we define the following statements.
\begin{description}
\item[$\RT^n_k$.] For any $P:[\N]^n \to k$, there exists an infinite set $H \subseteq \N$ such that $|P[[H]^n]|=1$. 
\item[$\RT^n$.] $\forall k \in \N \ \RT^n_k$.
\end{description}
\end{definition}

For any natural numbers $n$ and $k \geq 2$ it is straightforward to show directly both $\RT^n_{k+1}$ from $\RT^n_k $, and $\RT^n_2$ from $\RT^{n+1}_2$ in $\RCAo$.
%
%
Moreover, thanks to \cite{Jockusch,Specker,CJS}, it is well-known that in $\RCAo$ the situation is the following: 

\[
	\RT^1_2 < \RT^1 <  \RT^2_2 < \RT^2 < \RT^3_2 = \ACAo = \RT^3 = \dots = \RT^n  < \forall n \RT^n.
\]

Ramsey's Theorem for pairs in two colors has been proved to be not equivalent to any of the \vgt{Big Five} principles of Reverse Mathematics \cite{CJS, Liu}.\footnote{As remarked by Alexander Kreuzer there are other theorems know to be not equivalent to the Big Five before Ramsey's Theorem for pairs. For instance the Infinite Pigeonhole Principle \cite{Hirst} and the Weak Weak K\"{o}nig Lemma \cite{SimpsonYu}.} Moreover there exists a zoo of consequences of Ramsey's Theorem for pairs (from now on just Ramsey's Theorem). In this work we mainly focus on $\WRT$ (Weak Ramsey's Theorem) \cite{KeitaLNCS} which states that any coloring over the edges of a complete graph with countably many nodes admits an infinite homogeneous sequence. In order to analyse it we need to recall some more famous principles: $\CAC$ and $\ADS$. A chain is a totally ordered subset, and an antichain is a subset of pair-wise unrelated elements.

\begin{definition} Let $k \in \N$.
\begin{description}
\item[$\WRT^2_k$.] For any $P:[\N]^2 \to k$, there exist $c \in k$ and an infinite set $H=\bp{x_i : i \in \N} \subseteq \N$ such that for any $i\in\N$ $P(\bp{x_i,x_{i+1}})=c$.
\item[$\CAC$.] Every infinite poset has an infinite chain or antichain.
\item[$\ADS$.] Every infinite linear ordering has an infinite ascending or descending sequence.
\end{description}
\end{definition}

As shown in \cite{Hirschfeldt}, $\ADS$ is equivalent to the Transitive Ramsey Theorem for pairs ($\TRT$), stated below. A coloring $P: [\N]^2 \to k$ is said to be transitive if for any $x,y,z\in \N$,
\[
	P(\bp{x,y}) = P(\bp{y,z}) \implies P(\bp{x,z}) = P(\bp{x,y}). 
\]
The Transitive Ramsey Theorem for pairs states that every transitive coloring $P: [\N]^2 \to k$  admits an infinite homogeneous set.  It is straightforward to prove directly that $\CAC$ implies $\ADS$ in $\RCAo$. Moreover for any $k \in \N$, $\WRT^2_k$ lies between them.

\begin{proposition}[$\RCAo$]\label{Proposition: WRT-CAC-ADS} Let $k$ be a natural number. Then 
\begin{enumerate}
\item $\WRT^2_2 = \ADS$;
\item $\CAC \implies \WRT^2_k$.
\end{enumerate}
\end{proposition}
\begin{proof}
\begin{enumerate}
\item $\WRT^2_2 \implies \ADS$. Let $(\N,\prec)$ be an infinite linear ordering. Let $P:[\N]^2 \to 2$ be such that $P(\bp{x,y})=0$ if and only if $x \prec y$. Then, by $\WRT^2_2$ there exists an infinite homogeneous sequence.  If the sequence is in color $0$ we have an infinite increasing sequence; otherwise since the order is total we have an infinite decreasing sequence. 

$\ADS \implies \WRT^2_2$. Given $P:[\N]^2 \to 2$, let $P^*:[\N]^2 \to 2$ defined by induction over $|x-y|$ as follows
\begin{align*}
P^*(\bp{x,x+1}) &= P(\bp{x,x+1});\\
P^*(\bp{x,y}) &= 
\begin{cases}
P^*(\bp{x,z}) &\mbox{if } \exists z<y (x<z \wedge P^*(\bp{x,z})=P^*(\bp{z,y}))\\
P(\bp{x,y}) &\mbox{otherwise}.
\end{cases}
\end{align*}

A priori, the definition $P^*(\bp{x,y})=P^*(\bp{x,z})$ could assign more than one value to $P^*(\bp{x,y})$. However,
$P^*$ is well-defined since if there exist $x<z < z'<y$ such that 
\[
	P^*(\bp{x,z}) = P^*(\bp{z,y}) \neq P^*(\bp{x,z'}) = P^*(\bp{z',y})
\]
then either $P^*(\bp{z,z'})=P^*(\bp{x,z})$ or $P^*(\bp{z,z'})=P^*(\bp{z',y})$. If $P^*(\bp{z,z'})= P^*(\bp{x,z})$ by definition we get $P^*(\bp{x,z'}) = P^*(\bp{x,z})$, and this is a contradiction. Otherwise,
if $P^*(\bp{z,z'})= P^*(\bp{z',y})$, then $P^*(\bp{z,y}) = P^*(\bp{z',y})$ and this is once again a contradiction. It is straightforward to prove that $P^*$ is transitive. Since we assumed $\ADS$, by \cite{Hirschfeldt} we can use the Transitive Ramsey Theorem for pairs in two colors to derive the existence of an infinite homogeneous set $H^*$ for $P^*$ in color $c$. From $H^*$ we obtain the infinite homogeneous sequence $H$ for $P$. In fact let $\bp{x_n : n \in \omega}$ be an increasing enumeration of $H^*$. Define $H$ by: $x \in H$ if and only if $\exists n, m < x$ such that $x$ belongs to the minimum, with respect to the lexicographical order, of the shortest paths in color $c$ from $x_n$ to $x_m$. 

\item By induction over $k$. First of all we prove that $\RCAo \vdash \CAC \implies \WRT^2_2$. Given $P: [\N]^2 \to 2$, let $P^*$ as above and define define the poset $(\N, \prec)$, where
\[ 
	n \prec m \ \iff \ n < m \ \wedge \ P^*(\bp{n,m})=0.
\]
By $\CAC$ we have that there exists either an infinite chain or an infinite antichain. In the first case and since $\N$ is well-founded we obtain an infinite sequence in color $0$. In the second one we have an infinite homogeneous set in color $1$. Anyway we are done.

Assume now that $\RCAo \vdash \CAC \implies \WRT^2_k$ in order to prove that $\RCAo \vdash \CAC \implies \WRT^2_{k+1}$. Assume that $P: [\N]^2 \to k+1$ is given and define the poset $(\N, \prec)$ as above. Again by $\CAC$ we have that there exists either an infinite chain or an infinite antichain. In the former case we obtain an infinite sequence in color $0$ as well. In the latter one we have an infinite set $X$ whose nodes cannot be connected in color $0$.  Since $X$ is infinite, there exists a bijection between $X$ and $\N$. Therefore by applying ($\CAC$ and) the inductive hypothesis to $P \rest [X]^2$, we obtain an infinite homogeneous sequence. \qedhere
\end{enumerate}
\end{proof}

Hence we have the following situation
\[
	 \RCAo < \ADS = \TRT = \WRT^2_2 \leq  \WRT^2_3 \leq \dots \leq \WRT^2_k \leq \CAC < \RT^2_2 = \dots = \RT^2_k.
\] 
The equivalence between $\ADS$ and $\WRT^2_2$ was found independently by the first author and Patey \cite{Patey}. Since Lerman, Solomon and Towsner in \cite{LermanSolomonTowsner} proved that $\ADS < \CAC$ we get that at least one of the previous inequalities is strict. Apparently Patey \cite{Patey} separates $\CAC$ from $\WRT^2_k$. As far as we know the equivalence between $\WRT^2_k$ and $\WRT^2_{k+1}$ is still open.  The separation between $\CAC$ and $\RT^2_2$ was firstly proved by Hirschfeldt and Shore in \cite{Hirschfeldt}. 

Observe that, since the proof of Proposition \ref{Proposition: WRT-CAC-ADS}.2 is by induction over $k$, we cannot conclude from this proof that $\CAC \geq \forall k \WRT^2_k$ in $\RCAo$.

\section{Podelski and Rybalchenko's termination}\label{Section: TTandHc}

In \cite{Podelski} Podelski and Rybalchenko expressed the termination of transition-based programs as a property of well-founded relations. They said that a program $\tP$ is terminating if its transition relation restricted to the accessible states is well-founded. By using this definition they proved their main theorem which states that a program $P$ is terminating if it has a disjunctively well-founded transition invariant.
In this section we firstly introduce the Termination Theorem, then we analyse the notion of inductive well-foundedness in order to present the $H$-closure Theorem \cite{Hclosure}.

\subsection{Transition Invariants}
In this subsection we recall the definition of transition invariant and the Termination Theorem. For details we refer to \cite{Podelski}.

\begin{defi} As in \cite{Podelski}:
	\begin{itemize}
		\item A transition-based program $\tP = (S,I,R)$ consists of:
		\begin{itemize}
			\item $S$: a set of states,
			\item $I$: a set of initial states, such that $I \subseteq S$,
			\item $R$: a transition relation, such that $R \subseteq S \times S$.
		\end{itemize} 
		\item A computation is a maximal sequence of states $s_0,\ s_1,\ \dots$ such that
		\begin{itemize}
			\item $s_0\in I$,
			\item $(s_{i+1}, s_{i})\in R$ for any $i \in \N$.
		\end{itemize}
		\item The set $\acc$ of accessible states is the set of all the states which appear in some computation.
		\item The transition-based program $\tP$ is terminating if and only if $R \cap (\acc \times \acc)$ is well-founded.
		\item A transition invariant $T$ is a superset of the transitive closure of the transition relation $R$ restricted to the accessible states $\acc$. Formally,
				\[
					R^+ \cap (\acc \times \acc) \subseteq T.
				\]
		\item A relation $T$ is disjunctively well-founded if it is a finite union of well-founded relations; i.e. $T = T_0 \cup \dots \cup T_{k-1}$ where for any $i \in k$, $T_i$ is well-founded.
	\end{itemize}
\end{defi}

In this work each state is represented by a finite sequence $s$ which contains the values of the variables and the location of $s$ (the address of the instruction being executed). By using these definitions we can state the main result by Podelski and Rybalchenko. 

\begin{theorem}[Termination Theorem, Theorem 1 \cite{Podelski}] \label{Theorem: Podelski}
The transition-based program $\tP$ is terminating if and only if there exists a disjunctively well-founded transition invariant for $\tP$.
\end{theorem}

By unfolding definitions this result states that given a binary relation $R$, it is well-founded if and only if there exist a natural number $k$ and $k$-many well-founded relations $R_0,\dots, R_{k-1}$ whose union contains the transitive closure of $R$. The result is non trivial since the union of two well-founded relations can be ill-founded. As observed by Geser in \cite[pag 31]{Geser}, the fact that given any two well-founded binary relations if their union is transitive then it is well-founded has been remarked before Podelski and Rybalchenko. However the Termination Theorem is a non-trivial generalization of this result. In fact it cannot be directly proved from it by induction over the number of the relations, since we cannot keep the transitivity through the inductive steps.

\subsection{Inductively well-foundedness}

In order to introduce the $H$-closure Theorem we recall the definition of inductive well-foundedness as in \cite{Altenkirch}. A binary relation $R$ is classically well-founded (just well-founded for short) if it does not have infinite decreasing sequences. On the other hand $R$ is inductively well-founded if any $R$-inductive set contains all the elements in the domain of $X$, i.e. 
\[
	\forall x \forall X (\forall y (\forall z (zRy \implies z \in X)\implies y \in X) \implies x \in X). 
\]

Inductive well-foundedness is intuitionistically stronger than the classical notion \cite{BRBOUND}. Moreover even if classically they are equivalent, they are not the same in $\RCAo$.  In fact in this system we can prove that well-foundedness is stronger than  inductively well-foundedness, but the other implication is equivalent to $\ACAo$, as noticed by Marcone \cite{Marcone}.

\begin{lemma}[$\RCAo$]\label{lemma: wfimpliesiwf}
	$R$ well-founded implies $R$ inductively well-founded.
\end{lemma}
\begin{proof}
	 Assume by contradiction that $R$ is not inductively well-founded. Let $X$ be an inductive set such that $\N \setminus X$ is not empty and let $x \in \N \setminus X$. Then we can define by primitive recursion and minimization the following infinite decreasing $R$-sequence.
	\[
		f(n)= 
		\begin{cases}
		\mu y\ ( y R f(n-1)\ \wedge \ y \notin X) &\mbox{if } n>0\\
		x &\mbox{if } n=0. \qedhere
		\end{cases} 
	\]
\end{proof}

\begin{proposition}[$\RCAo$, \cite{Marcone}]\label{proposition: wfiwf}
The following are equivalent:
\begin{itemize}
\item $\ACAo$ ;
\item for any binary relation $R$, $R$ inductively well-founded implies $R$ well-founded.
\end{itemize}
\end{proposition}
\begin{proof}
\vgt{$\Downarrow$}. Assume that $\ACAo$ holds and assume that $R$ is not well-founded. Then there exists an infinite decreasing $R$-sequence $f$. By $\ACAo$ we can define $X$ to be the range of $f$. Then $\overline{X} =\N \setminus X$ is an inductive set which witnesses that $R$ is not inductively well-founded.
 
\vgt{$\Uparrow$}.Assume that for any binary relation $R$, $R$ inductively well-founded implies $R$ well-founded. Given a 1-1 function $f$ we have to prove that $f$ has a range. Define the binary relation $R$ on the set $\bp{a_n: n \in \N} \cup \bp{b_n: n \in \N}$, with $a_n \neq b_m$ for any $n,m$, as follows:
\begin{align*}
b_n R a_k &\iff f(n)=k\\
a_k R b_n &\iff f(n+1)=k
\end{align*}
$R$ is not well-founded as witnessed by the descending sequence $ a_{f(0)}, b_0, a_{f(1)}, b_1, \dots$.  Indeed for any $i \in \N$, $a_{f(i+1)} R b_i$ and $b_{i+1} R a_{f(i+1)}$. Then by hypothesis it is not inductively well-founded. Hence there exists $X \subseteq \N$ inductive such that $\overline{X}= \N \setminus X$ is not empty. Therefore $\overline{X}$ contains no $a_k$ such that $k$ is not in the range of $f$, since all these elements have no $R$-predecessors. Moreover there exists $m\in \N$ such that $\overline{X}$ contains every $a_{f(n)}$ with $n\geq m$. In fact since $\overline{X}$ is non-empty we have two possibilities: either $b_i \in \overline{X}$ for some $i$ and then for any $n\geq i+1$ we have $a_{f(n)}\in \overline{X}$, or $a_{k}\in \overline{X}$ for some $k$, but since it does not belong to $X$ also its predecessor belongs to $\overline{X}$ and it is $b_i$ for some $i$, therefore we are done again. 
Hence we can define the range of $f$ as follows:
\[
	\bp{n \in \N : a_n \in X} \cup \bp{f(n): n < m} \qedhere
\]
\end{proof}

However as witnessed by Lemma \ref{lemma: wfimpliesiwf}, the classical definition of well-foundedness implies the inductive definition and moreover if the binary relation is transitive then also the other implication turns out to be true. 

\begin{lemma}[$\RCAo$]
If $R$ is transitive and inductively well-founded, then it is well-founded.
\end{lemma}
\begin{proof}
Assume by contradiction that $R$ is not well-founded, then there exists an infinite decreasing transitive $R$-sequence:
$f: \N \rightarrow \N$. 
	
	We have two possibilities:
	\begin{enumerate}
	\item $f$ is not injective. Let $n$, $m$ such that $n < m $ and $f(n)=f(m)$. Then define 
	\[
	\begin{split}
		g: \N &\rightarrow \N \\
			k &\mapsto f(n+(k \mod(m-n))).
	\end{split}
	\]
	\item $f$ is injective.  Then define $h: \N \rightarrow \N$ as follows
	\[
		h(k)= 
				\begin{cases}
				    f(0) &\mbox{if } k=0;\\
					\mu m ( m > h(k-1) \wedge  \forall i < m (f(m) > f(i))) &\mbox{if } k>0.
				\end{cases} 
	\]
	Let $g: \N \rightarrow \N$ be such that	$g(k)= f(h(k))$.
	\end{enumerate}
	The function $g$ is a $R$-sequence. In the case (1), we have 
	\[
		g(0) = f(n), g(1) = f(n+1),\dots , g(m-n-1) = f(m-1), g(m-n) = f(n) = f(m) = g(0),
	\] 
	therefore 
	\[
		g(0) R g(1) R \dots R g(m-n-1) R g(m-n) = g(0).
	\]
	In the case (2), $g$ is an infinite subsequence of a decreasing transitive $R$-sequence, therefore $g$ is a decreasing $R$-sequence.
	
	There exists $X$ which is the range of $g$. In the case (1) the set $X$ exists because it is finite. In the case (2), $g$ is increasing, therefore $\exists m(g(m)=n)$ is equivalent to
	\[
		 \exists m \leq n (g(m) = n),
	\]
	which is $\Delta^0_1$.
	Let $\overline{X}= \N \setminus X$. $\overline{X}$ is different from $\N$: if we prove that $X$ is $R$-inductive we conclude that $R$ is not inductive.  In fact
	\[
		\forall y(\forall z (z \succ y \implies z \in \overline{X})\implies y \in \overline{X})
	\]
	holds. We prove the contrapositive. If $y\notin \overline{X}$, then $y\in X$ and since $g$ is an infinite decreasing $R$-sequence, this implies there exists $z$ in $X$ such that $z \succ y$.
\end{proof}

\subsection{H-closure Theorem}

Podelski and Rybalchenko in \cite{Podelski} classically proved their Termination Theorem by using Ramsey's Theorem for pairs which is a purely classical result \cite{RT22iff3LLPO}. However if we consider the definition of inductive well-foundedness instead of the classical one the Termination Theorem turns out to be intuitionistic. $H$-closure was introduced for this purpose \cite{Hclosure}. It was not the first intuitionistic proof of the intuitionistic version  of the Termination Theorem since it was already proved by Vytiniotis, Coquand and Wahlstedt  in \cite{CoquandStop} by using the Almost-Full Theorem \cite{Coquand}. However the Almost-Full Theorem and the $H$-closure Theorem are not intuitionistically equivalent, the first one is obtained from Ramsey's Theorem by two classical steps: a contrapositive and an application of the De' Morgan Law. On the other hand $H$-closure is obtained from Ramsey's Theorem by using just a contrapositive. Moreover both of them provide the fragment of Ramsey's Theorem for pairs needed to prove the Termination Theorem.

First of all we have to define $H$-well-foundedness as in \cite{Hclosure}. We say that a relation is $H$-well-founded if it has no infinite transitive decreasing sequences. Let $\succ$ be the one-step expansion between finite sequences: i.e.
\[
	\ap{y_0, \dots, y_{k}} \succ \ap{x_0, \dots, x_h} \iff \ k=h+1 \wedge \forall i \in k (y_{i}=x_i). 
\]
Here we say that a sequence $a_0,a_1,\dots, a_n$ is $\succ$-decreasing sequence if  $a_n \succ \dots \succ a_1 \succ a_0$. 

\begin{defi}
 	Let $R$ be a binary relation on $S$. 
 	\begin{itemize}
 	\item $H(R)$ is the set of the $R$-decreasing transitive finite sequences on $S$:
 		\[
 			 \ap{x_0,\dots, x_{n-1}} \in H(R) \iff \forall i, j < n\ (i<j \implies x_j R x_i).
 	 	\]
 	\item $R$ is (inductively) $H$-well-founded if $H(R)$ is  (inductively) $\succ$-well-founded.
 	\end{itemize}
\end{defi}

Observe that there is an infinite $\succ$-decreasing sequence in $H(R)$ if and only if there is an infinite $R$-decreasing transitive sequence. There is a strong connection between (inductively) well-foundedness and (inductively) $H$-well-foundedness, as shown by the following proposition.

\begin{proposition}[Proposition 1 \cite{Hclosure} ]\label{proposition: wfHwf}
Let $R$ be a binary relation.
	\begin{enumerate}
		\item If $R$ is (inductively) well-founded then $R$ is (inductively) $H$-well-founded;
		\item If $R$ is (inductively) $H$-well-founded and $R$ transitive then $R$ is (inductively) well-founded.
	\end{enumerate}
\end{proposition}

Moreover there are relations which are $H$-well-founded but not well-founded. For instance $R=\bp{\ap{n+1,n}: n \in \N}$. In fact any sequence in $H(R)$ has length at most two, but $R$ is not well-founded.

The $H$-closure theorem states that the inductively $H$-well-founded relations are closed under finite unions. Formally

\begin{theorem}[Theorem 2 \cite{Hclosure}]\label{theorem: Hclosure}
Let $R_0,\dots, R_{k-1}$ be binary relations. If $R_0, \dots, R_{k-1}$ are inductively $H$-well-founded then $(R_0 \cup \dots \cup R_{k-1})$ is inductively $H$-well-founded.
\end{theorem}

\section{Termination Theorem and H-closure Theorem in the reverse mathematics' zoo}\label{Section: zoo}

The first question about the Termination Theorem and the $H$-closure Theorem is whether they are equivalent to Ramsey's Theorem for pairs over $\RCAo$. In this section we prove that actually the H-closure Theorem is equivalent to Ramsey's Theorem for pairs. On the other hand the Termination Theorem is equivalent to Weak Ramsey's Theorem.

\subsection{H-closure Theorem and Ramsey's Theorem for pairs}

The main problem we have in order to prove this equivalence is that the $H$-closure Theorem uses the inductive definition of well-foundedness which is not equivalent to the classical one as shown by Proposition \ref{proposition: wfiwf}. Fortunately we can show that if $R$ is inductively $H$-well-founded then it is classical $H$-well-founded as well. And by using this result we obtain the desired equivalence. 

Let $\Seq$ be the set of code for finite sequences. First of all we can observe that $H(R)$ is defined by the following formula
	\[
	  	s\in H(R) \iff  s\in \Seq \wedge \forall i,j < \lh(s) (i<j \implies s(j) R s(i)).
	\]
	
\begin{lemma}[$\RCAo$] \label{iHwfimpliesHwf} Let $R$ be a binary relation.
	If $R$ is inductively $H$-well-founded then $R$ is $H$-well-founded. 
\end{lemma}
\begin{proof}
	Assume by contradiction that $R$ is not $H$-well-founded. Then there exists an infinite decreasing $\succ$-sequence $f: \N \rightarrow H(R)$.  Let $X = \rng(f)$, it exists since $s \in \rng(f)$ iff $\exists x \leq \lh(s)(f(x)=s)$. Hence put $\overline{X}= H(R) \setminus X$. This is a counterexample to the inductiveness of $\succ$ in $H(R)$. In fact
	\[
		\forall y(\forall z (z \succ y \implies z \in \overline{X})\implies y \in \overline{X})
	\]
	holds (where $z \succ y$ iff $s(z) \succ s(y)$).  We prove the contrapositive. If $y\notin \overline{X}$, then $y\in X$ and since $f$ is an infinite decreasing $\succ$-sequence, this implies there exists $z$ in $X$ such that $z \succ y$.
\end{proof}

Thanks to the previous lemma we can prove that the $H$-closure Theorem is equivalent to Ramsey's Theorem for pairs  by considering the classical definition of well-foundedness instead of the intuitionistic one.

\begin{theorem}[$\RCAo$]\label{theorem: HCL-RT}
Let $k$ be a natural number, then the following are equivalent
\begin{enumerate}
\item the $H$-closure Theorem for $k$-many relation, i.e. the union of $k$-many inductively $H$-well-founded relations is inductively $H$-well-founded;
\item the union of $k$-many $H$-well-founded relations is $H$-well-founded;
\item $\RT^2_k$. 
\end{enumerate}
\end{theorem}

\begin{proof}
\vgt{$1\Leftrightarrow 2$}. It follows from Lemma \ref{lemma: wfimpliesiwf} and Lemma \ref{iHwfimpliesHwf}. 

\vgt{$2 \Rightarrow 3$}. Let $R_i'$ be symmetric relations defined by a $k$-coloring on $[\N]^2$ such that
		\[
			R_0' \cup \dots \cup R_{k-1}' =\bp{(x,y) \in \N \times \N : x \neq y}.
		\]
		We need to prove that there exists an infinite homogeneous set $X \subseteq \N$.
		For any $i < k$, put
		\[
			R_i := \bp{(x,y) : x R_i' y \wedge x>y}
		\]
		$R_i$ is defined by $\Delta^0_1$-comprehension.
		Then
		\[
			R_0 \cup \dots \cup R_{k-1} =\bp{(x,y) : x > y },
		\]
		and $\bp{n: n \in \N}$ is an infinite transitive decreasing $(R_0 \cup \dots \cup R_{k-1})$-sequence. By applying the $H$-closure Theorem we obtain there exists an infinite transitive decreasing sequence $f_i: \N \rightarrow H(R_i)$ for some $i < k$.

		Define $\tilde{f}_i: \N \to \N$ such that $\tilde{f}_i(n)$ is the last element of $f_i(n)$, i.e. $\tilde{f}_i(n)= f_i(n)(n-1)$. Let $X$ be the range of $\tilde{f}_i$. It exists since $\tilde{f}_i$ is increasing. Then $X$ is an infinite homogeneous subset of $\N$.
		
\vgt{$3\Rightarrow 2$}. Suppose that there exists an infinite transitive decreasing $(R_0 \cup \dots \cup R_{k-1})$-sequence: 
$f: \N \rightarrow \N$. For any $i < k $, put
		\[
			R_i' := \bp{(m,n) : (m<n \wedge f(n) R_i f(m) ) \vee (n<m \wedge f(m) R_i f(n))},
		\]
		Since $f$ is $(R_0 \cup \dots \cup R_{k-1})$-transitive for any $m,\ n \in \N$ we have
		\[
			m<n \implies f(n) (R_0 \cup \dots \cup R_{k-1}) f(m),
		\]
		then $(R_0' \cup \dots \cup R_{k-1}') = \bp{(m,n) \in \N \times \N : m \neq n}$.
		Thanks to $\RT^2_k$ there exists an infinite homogeneous set $X \subseteq \N$ for some $R_i'$, for some $i< k$. Then 
		\[
			\forall m, n \in X (m<n \implies n R_i' m).
		\]
		Then define $h: \N \rightarrow \N$ as follows
			\[
				h(k)= 
						\begin{cases}
						   	\mu m (m \in X) &\mbox{if } k=0;\\
							\mu m (m \in X \wedge m > h(k-1)) &\mbox{if } k>0.
						\end{cases} 
			\]
Hence for any $m < n$, we have $h(m)<h(n)$ and since $h(m), h(n) \in X$ we get $h(n) R_i' h(m)$. Therefore $h$ is an infinite transitive decreasing $R_i'$-sequence. By definition of $R_i'$, $h$ is also an infinite decreasing $R_i$-sequence and therefore $H(R_i)$ is ill founded.
\end{proof}

\subsection{Termination Theorem, Weak Ramsey's Theorem and Weak $H$-closure}

In this subsection we are going to prove that the Termination Theorem is equivalent in $\RCAo$ to Weak Ramsey's Theorem. Moreover we can observe that these two results  are equivalent to a weak version of $H$-closure, which we call Weak $H$-closure.

\begin{theorem}[Weak $H$-closure Theorem.]
Given any relations $R_0, \dots, R_{k-1}$ with $k \in \N$, if $R_i$ is well-founded for every $i<k$ then  $\bigcup\bp{R_i : i < k}$
 is $H$-well-founded.
\end{theorem}

This result follows from $H$-closure Theorem by applying Proposition \ref{proposition: wfHwf}.1. Observe that thanks to Lemma \ref{iHwfimpliesHwf} the theorem above is equivalent to the statement ``the union of well-founded relations is inductively $H$-well-founded''. While the statement ``the union of inductively-well-founded relations is (inductively) $H$-well-founded'' is stronger by Proposition \ref{proposition: wfiwf}.

\begin{theorem}[$\RCAo$]\label{theorem: PRT-WRT}
	Let $k$ be a natural number. Then the following are equivalent:
	\begin{enumerate}
	\item the Termination Theorem for transition invariants composed by $k$-many relations: given a binary relation $R$, if there exist $k$-many well-founded relations whose union contains the transitive closure of $R$, then $R$ is well-founded; 
	\item $\WRT^2_k$;
	\item Weak $H$-closure Theorem for $k$-many relations: the union of $k$-many well-founded relations is $H$-well-founded.
	\end{enumerate}
\end{theorem}

\begin{proof}
\vgt{$1\Rightarrow 2$}. Let $P: [\N]^2 \to k$ be a coloring. For any $i < k$ define a binary relation $R_i$ as follows:
\[
	x R_i y \iff (x > y)  \wedge P(\bp{x,y})= i.
\] 
Assume by contradiction that there are no infinite sequences for any $R_i$. Put $R= \bp{(x+1,x): x \in \N}$ then $\bigcup\bp{R_i : i<k } = R^+ = \bp{(x,y) : x > y}$. Then by applying the Termination Theorem $R$ is well-founded. Contradiction. 

\vgt{$2 \Rightarrow 3$}. Assume that $R_0, \dots, R_{k-1}$ are well-founded and suppose by contradiction that there exists $f: \N \to \N$ such that it is a transitive decreasing $R_0 \cup \dots \cup R_{k-1}$-sequence. For any $i < k$ define $R_i'$ as in the proof of Theorem \ref{theorem: HCL-RT}. By using Weak Ramsey's Theorem instead of Ramsey's Theorem we obtain that $h$ (defined as in the proof of Theorem \ref{theorem: HCL-RT}) is an infinite decreasing sequence for $R_i'$ for some $i<k$. This is a contradiction with the fact that $R_i$ is well-founded.

\vgt{$3\Rightarrow 1$}. Assume that $R_0, \dots, R_{k-1}$ are well-founded and that
\[
	\bigcup\bp{R_i : i < k} \supseteq R^+,
\]
in order to prove that $R$ is well-founded. Assume by contradiction that $R$ is non well-founded. Then $R^+$ is not well-founded, hence since it is transitive, it is not $H$-well-founded. Then also $\bigcup\bp{R_i : i < k}$ is not $H$-well-founded as well. Then by Weak $H$-closure there exists $i< k$ such that $R_i$ is not well-founded. Contradiction.
\end{proof}

\begin{remark}
By using the same argument we can prove that for any natural number $k$ the following are equivalent over $\RCAo$.
\begin{itemize}
\item For any binary relation $R$, if there exist $k$-many inductively well-founded relations whose union contains the transitive closure of $R$ then $R$ is well-founded.
\item The union of $k$-many inductively well-founded relations is $H$-well-founded.
\end{itemize}
\end{remark}

Observe that  by Proposition \ref{Proposition: WRT-CAC-ADS}.2 and since $\RT^2_2 > \CAC$ \cite{Hirschfeldt}, we have that the Termination Theorem for transition invariant composed of $k$-many relations is strictly weaker that Ramsey's Theorem for pairs and $k$-many colors. Hence we can provide an answer to \cite[Open Problem 2]{Gasarch} posed by Gasarch. There is no program proved to be terminating by the Termination Theorem, such that the proof of this fact requires the full Ramsey Theorem for pairs.

Moreover notice that $\forall k \WRT^2_k$ is provable from $\CAC$ plus the full induction. On the other hand $\CAC$ plus the full induction does not imply $\RT^2_2$ (and with more reason $\forall k \RT^2_k$), since the separation between $\CAC$ and $\RT^2_2$ provided in  \cite{Hirschfeldt} is done over $\omega$-models \footnote{Models whose first order part is standard (e.g. \cite{SOSOA}).}, which always enjoy the full induction. We conclude that $\forall k \WRT^2_k$ does not imply $\forall k \RT^2_k$ (even $\RT^2_2$). Thanks to Theorem \ref{theorem: PRT-WRT}, we get a negative answer to \cite[Open Problem 3]{Gasarch} posed by Gasarch:  is the Termination Theorem equivalent to full Ramsey Theorem for pairs? In fact the full Termination Theorem is equivalent to the full Weak Ramsey Theorem which is strictly weaker than the full Ramsey Theorem over $\RCAo$.

\section{Weight functions, bounds, and $H$-bounds}\label{Section: weightbounds}
In the study of the termination analysis, it is important to investigate a bound for the number of steps required by a program to terminate by analysing the structure of the program. For this purpose, we need a formal notion of bounds.

\begin{defi}\label{def:bounds}
Let $R$ be a binary relation on $S$.
\begin{itemize}
	\item A weight function for $R$ is a function $f: S \to \N$ such that for any $x,\ y \in S$
	\[
		x R y \implies f(x) < f(y).
	\]
	We say that $R$ has height $\omega$ if there exists a weight function for $R$.
    \item A bound for $R$ is a function $f:S\to \N$ such that for any $R$-decreasing sequence $\langle a_{0}, \dots, a_{l-1} \rangle$, $l\le f(a_{0})$, \textit{i.e.}, any decreasing $R$-sequence starting from $a$ is shorter than $f(a)$.
    \item A $H$-bound for $R$ is a function $f:S\to \N$ such that for any $R$-decreasing transitive sequence $\langle a_{0}, \dots, a_{l-1} \rangle$, $l\le f(a_{0})$, \textit{i.e.}, any decreasing transitive $R$-sequence starting from $a$ is shorter than $f(a)$.
  \end{itemize}
\end{defi}

It is easy to see that in $\ACAo$,  $R$ has a bound if and only if $R$ has a weight function. However one of the implications cannot be proved in $\RCAo$.

\begin{prop}[$\ACAo$]
Given a binary relation $R$. $R$ has a bound if and only if $R$ has a weight function.
\end{prop}
\begin{proof}
\vgt{$\Rightarrow$}. If $R$ has a bound $f:S \to \N$ then we can define $f^*:S \to \N$ as follows
\[
	f^*(x) = \max \{ l : l \mbox{ is the length of a decreasing $R$-sequence from } x\}.
\]
For any $x \in S$, $f^*(x) \in \N$ since $f^*(x) \leq f(x)$ and $f^*$ is a weight function by definition. 

\vgt{$\Leftarrow$}. If $R$ has a weight function $f:S \to \N$, then if $\langle a_i : i\in l \rangle$ is a decreasing $R$-sequence, $f(a_i) \geq l$. 
\end{proof}

The second implication is in $\RCAo$, while the first one requires $\Pi^0_1$-comprehension. If we assume that $R$ is finitely branching, i.e. there exists $\delta:S \to \mathcal{P}_{{<}\omega}(S)=\bp{X \subseteq S : |X|<\omega}$ such that $xRy$ if and only if $x \in \delta(y)$, then also the first implication turns out to be provable in $\RCAo$. In fact 
\[
	f^*(x) = \max \{ l : \delta^l(x)\neq \emptyset\},
\]
where $\delta^{l+1}(x) = \bigcup\{\delta(y): y \in \delta^{l}(x)\}$. This set exists for any $l$ since it is a finite set. As above $f^*(x)$ is bounded by $f(x)$, therefore $f^*(x) \in \N$.

In the general case, we have the following.
\begin{thm} \label{theorem: weightbounds}
The following are equivalent over $\RCAo$.
\begin{enumerate}
 \item $\WKLo$.
 \item For any relation $R\subseteq S^{2}$, $R$ has a bound if and only if $R$ has a weight function.
\end{enumerate}
\end{thm}
\begin{proof}
We reason within $\RCAo$.
Without loss of generality, we may assume that $S=\N$.
Let $S_{n}=\{0,\dots,n-1\}$ and $R_{n}=R\cap {S_{n}}^{2}$ then the pair $S_{n}, R_{n}$ associates two finite subsets of $S$ and $R$ respectively.
One can easily check the following over $\RCAo$:
\begin{itemize}
 \item $f:S\to \N$ is a weight function on $S$ if and only if $f\rest S_{n}$ is a weight function on $S_{n}$ for every $n\in\N$.
 \item $f:S\to \N$ is a bound on $S$ if and only if $f\rest S_{n}$ is a bound on $S_{n}$ for every $n\in\N$.
 \item If $R_{n}$ has a bound $h:S_{n}\to\N$, then $R_{n}$ has a weight function $f:S_{n}\to\N$ such that $f\le h$ (as the above proposition).
\end{itemize}

\vgt{$\Downarrow$}. Assume that $R\subseteq \N^{2}$ has a bound $h:\N\to\N$.
Define a tree $T\subseteq\N^{<\N}$ as follows:
\[ \sigma\in T\iff \lh(\sigma)=n \wedge \sigma:S_{n}\to \N\text{ is a weight function on $S_{n}$}\wedge \A k\, \sigma(k)\le h(k).\]
Then, by the last point above, this $T$ is infinite.
Thus, by bounded K\"onig's Lemma \footnote{Bounded K\"onig's Lemma is equivalent to $\WKLo$ \cite[Lemma IV.1.4]{SOSOA}.}, $T$ has an infinite path $f:\N\to\N$ such that $f\le h$.
This $f$ is a weight function for $R$.

\vgt{$\Uparrow$}. We show (the restricted version of) $\Sigma^{0}_{1}$-separation.\footnote{$\Sigma^{0}_{1}$-separation is equivalent to $\WKLo$ \cite[Lemma IV.4.4]{SOSOA}.}
Let $p,q:\N\to\N$ be one-to-one functions such that $\rng(p)\cap\rng(q)=\emptyset$.
We want to find a set $X$ such that $\rng(p)\subseteq X\subseteq \N\setminus\rng(q)$.
Let $S=\N\times 4$. We claim there is some relation $R\subseteq S^{2}$ such that $(p(n),3) R (n,1) R (p(n),0)$ and $(q(n),0) R (n,2) R(q(n),3)$, and there is no other relation. We may prove in $\RCAo$ that $R$ exists because $R$ has a $\Delta^0_1$ definition:
\begin{align*}
 (n,i)R(m,j)\iff& (p(m)=n\wedge (i,j)=(3,1))\vee(p(n)=m\wedge (i,j)=(1,0))\\
& \vee (q(n)=m\wedge (i,j)=(2,3))\vee(q(m)=n\wedge (i,j)=(0,2)),
\end{align*}
i.e. if $p(n)=m$ then $(m,3)R(n,1)R(m,0)$, if $q(n)=m$ then $(m,0)R(n,2)R(m,3)$, and there is no other relation.
All $R$-sequences have at most three elements, because no element of $S$ has both the form $p(n)$ and the form $q(n’)$, for any $n, n’ \in \N$. Put $h:S\to\N$ as $h((n,i))=2$, then $h$ is a bound for $R$.
By (2), take a weight function $f:S\to \N$.
If $m$ is in $\rng(p)$ then $m = p(n)$ for some $n$ and there is a $R$-sequence $(m,3) R (n,1) R (m,0)$, hence $f((m,3)) < f((m,0))$. If $m$ is in $\rng(q)$ then $m = q(n)$ for some $n$ and there is a $R$-sequence $(m,0) R (n,2) R (m,3)$, hence $f((m,0)) < f((m,3))$. Thus, $X=\{m: f((m,3))<f((m,0))\}$ is a set separating $\rng(p)$ and $\rng(q)$.
\end{proof}

\section{Termination analysis and proof-theoretic strength}\label{Section: application}
In this section we apply the result we obtained about the reverse mathematical strength of the Termination Theorem to get bounds. In order to do that we need to recall some classical results. As standard we denote $\Sigma^0_1$-induction as $\II$ and with $\BII$ the bounding principle for $\Sigma^0_2$-formulas \cite{HirschfeldtRM}.

\begin{thm}[Parsons 1970, see e.g., \cite{prov-rec-func}] \label{theorem: parson}
The class of provable recursive functions of $\II$ is exactly the same as the class of primitive recursive functions.
\end{thm}

\begin{thm}[Paris/Kirby\cite{Paris-Kirby}]
$\BII$ is a $\Pi^{0}_{3}$-conservative extension of $\II$.
\end{thm}

\begin{thm}[Chong/Slaman/Yang\cite{CSY2012}]
$\WKLo+\CAC$ is a $\Pi^{1}_{1}$-conservative extension of $\BII$.
\end{thm}

Thus, we have the following.
\begin{cor}
The class of provable recursive functions of $\WKLo+\CAC$ is exactly the same as the class of primitive recursive functions.
\end{cor}

On the other hand, by Theorem~\ref{theorem: PRT-WRT} and Proposition \ref{Proposition: WRT-CAC-ADS}, we have the following.
\begin{prop}
The following is provable within $\WKLo+\CAC$.
\begin{itemize}
 \item[$(\dag)$] any relation $R$ for which there exists a disjunctively well-founded transition invariant composed of $k$-many relations $R_1 \cup \dots \cup R_{k-1} \supseteq R^+$ with bounds is well-founded.
\end{itemize}
\end{prop}

Consider now the special case of $(\dag)$ in the real world: if $R$ is a primitive recursive relation generated by a primitive recursive transition function  (in particular it is deterministic), and each $R_{i}$ has a primitive recursive bounds. Note that any primitive recursive function is strongly represented within $\RCAo$. Then, $(\dag)$ (together with $\WKLo$) means that
\begin{itemize}
 \item[$(\dag\dag)$] for any state $a$, there exists a bound $b\in\N$ of $R$-sequences from $a$.
\end{itemize}
Since $(\dag\dag)$ is a $\Pi^{0}_{2}$-statement provable in $\WKLo+\CAC$, the function $a\mapsto b$ must be bounded by a primitive recursive function.
Thus, we have the following.
\begin{cor}\label{cor:PRB}
Any relation generated by a primitive recursive transition function for which there exists a disjunctively well-founded transition invariant composed of $k$-many relations with primitive recursive bounds has a primitive recursive bound.
\end{cor}

Observe that since we worked in $\WKLo$ and thanks to Theorem \ref{theorem: weightbounds} this is another version of the result obtained in \cite{HBOUND2} by using the constructive proof of the Termination Theorem. 

Let us provide a simple example. 
\begin{example} Consider the following transition-based program.
\begin{verbatim}
while (x > 0 AND y > 0)
    if(x > y)
          (x, y) = (y, 2^{x+y})  (1)
    else
          (x, y) = (x, y - 1)    (2)
\end{verbatim}

where $x$, $y$ have domain all integers. A transition invariant for this program is $R_1 \cup R_2$, where 
\begin{align*}
R_1 &:= \bp{(\ap{x,y},\ap{x',y'}) : x> 0 \wedge x' < x}\\
R_2 &:= \bp{(\ap{x,y},\ap{x',y'}) : y> 0 \wedge y' <y }
\end{align*}

In fact if $\ap{x',y'}R^+\ap{x,y}$, by definition of transitive closure there exists a finite number of $R$-steps between them. Moreover notice that $x$ is weakly decreasing. Hence if one of these steps is a (1)-step, then $(\ap{x,y},\ap{x',y'}) \in R_1$, otherwise any step is a (2)-step, hence $y$ decreases everywhere and so $(\ap{x,y},\ap{x',y'}) \in R_2$.

Since each $R_i$ is primitive recursive bounded and thanks to Corollary \ref{cor:PRB}, $R$ has a primitive recursive bound.
\end{example}

Our goal is to characterize the relationship between properties of the transition invariants and how many steps are required by a program to terminate. In the next sections we will study the reverse mathematical strength of the $H$-closure Theorem and the Termination Theorem for relations of height $\omega$ and for relations with bounds and $H$-bounds, strengthen investigations about bounds. 

\section{Bounded versions of Termination Theorem}\label{Section: boundedtermination}

The goal of this section is to study the strength of some bounded versions of the Termination Theorem. They turn out to be equivalent to suitable versions of Paris and Harrington's Theorem. Paris and Harrington's Theorem is a strengthened version of finite Ramsey's Theorem which is unprovable in Peano Arithmetic, since it implies the consistency of Peano Arithmetic \cite{PH}. For all bounded versions proposed in this section the \vgt{bounded Termination Theorem} and the \vgt{bounded $H$-closure Theorem} turn out to be equivalent.

Since Paris Harrington's Theorem for pairs and $k$ many colors is provable within $\RCAo$ for any $k$, throughout this section we work in the subsystem $\RCAs$, defined for the language of second order arithmetic enriched with an exponential operation (e.g. \cite{SOSOA}). $\RCAs$ consists of the basic axioms together with the exponentiation axioms (elementary function arithmetic), $\Delta^0_0$ induction and $\Delta^0_1$-comprehension.

We denote with $\Ack_k$ the usual $k$-class of the Fast Growing Hierarchy \cite{Lob}. Define
\begin{align*}
\begin{cases}
 F_{0}(x)=x+1,\\
 F_{n+1}(x)=F_{n}{}^{(x+1)}(x).
\end{cases}
\end{align*}

Then $\Ack_k$ is the closure under limited recursion and substitution of the set of functions defined by constant, projections, sum and $F_h$ for any $h \leq k$. 

%

\subsection{Termination Theorem for relations of height $\omega$} \label{Section: heighomega}
In here we present the equivalence in $\RCAs$ between the Termination Theorem for relations of height $\omega$, $H$-closure for relations of height $\omega$ and the principle we call Weak Paris Harrington Theorem. This equivalence holds level by level: i.e. we prove that $H$-closure for $k$-many relations of height $\omega$ is equivalent to the Termination Theorem for $k$-many relations of height $\omega$ and to the Weak Paris Harrington Theorem for $k$-many colors.

First of all we state the theorems we deal with. We say that a set $X$ is $1$-\emph{large} if $\min X < |X|$. Given a coloring $P:[X]^{2}\to k$, a set $Y \subseteq X$ is \emph{weakly homogeneous} if its increasing enumeration is a homogeneous sequence for $P$. Therefore if $Y = \bp{y_0 < y_1 < \dots < y_n < \dots }$, there exists $i \in k$ such that $P(\bp{y_n, y_{n+1}})=i$ for all $n$.
 
\begin{defi}\label{def1}
For given $h, k\in\N$, we define the following statements.
\begin{enumerate}
 \item $\PH_{k}^{h,2}$: Given $f:\N \to \N$ such that for all $n\in \N$ $f(n+1)< F_{h}(f(n))$, for all coloring $P:[\ran(f)]^{2}\to k$, there exists a homogeneous set for $P$ which is $1$-large.
 \item $\WPH_{k}^{h,2}$: Given $f:\N \to \N$ such that for all $n\in \N$ $f(n+1)< F_{h}(f(n))$, for all coloring $P:[\ran(f)]^{2}\to k$, there exists a weakly homogeneous set for $P$ which is $1$-large.
 \item $k$-$\HCT^{h}$: If, for all $i<k$, $R_{i}$ is a binary relation of height $\omega$ whose weight function $f_i$ is such that $f_i(n)<F_h(n)$ for any $n$, then any transitive $R_{0}\cup\dots\cup R_{k-1}$-decreasing sequence $f:\N \to \N$ such that for all $n\in \N$ $f(n+1)< F_{h}(f(n))$ is finite. 
 \item $k\text{-}\TT^{h}$: Let $R$ be a deterministic binary relation, whose transition function $f:\N \to \N$ is such that for any $n\in \N$ $f(n+1)< F_{h}(f(n))$. If there exists a disjunctively well-founded transition invariant for $R$ composed of $k$-many relations of height $\omega$ whose weight functions $f_i$ are such that $f_i(n)<F_h(n)$ for any $n$, $R$ is well-founded.
\end{enumerate}
\end{defi}

By using these definitions we can prove the following.
 
\begin{thm}[$\RCAs+ \Tot(F_h)$]\label{thm-1}
For any natural number $k$,
\[
	\WPH^{h,2}_{k} = k\text{-}\HCT^h =  k\text{-}\TT^h.
\]
\end{thm}
\begin{proof}
\vgt{$\WPH_{k}^{h,2}\Rightarrow k\text{-}\HCT^h$}: Let binary relations $\bp{ R_{i}\subseteq S^{2}: i<k }$ and weight functions $\bp{f_{i} : S\to \N \ : i<k}$ be given as in the hypotheses of $k\text{-}\HCT^h$.
For the sake of contradiction, assume $R = \bigcup\{R_{i} : i < k\}$ and let $\langle a_{j}: j\in\N \rangle$ be an infinite transitive sequence for $R$ such that $a_{i+1}< F_h(a_i)$.
Define $f: \N \to \N$ as
\[
	f(n)=\max\bp{F_{h}(a_{j})+(n-j+1): j \leq n}.
\]
Observe that for any $n$, $f(n) \geq F_{h}(a_n)+1$, hence:
\begin{align*}
f(n+1) &= \max\bp{f(n)+1, F_{h}(a_{n+1})+1} < \max\bp{f(n)+1, F_{h}(F_h(a_n))+1}\\
&\leq \max\bp{f(n)+1, F_{h}(F_h(a_n)+1)} \leq F_h(f(n)).
\end{align*}
Since $f$ is increasing, we can define within $\RCAo$ $X=\{f(j) : j\in\N\}$. Define $P:[X]^{2}\to k$ as $P(\bp{f(n),f(m)})=\min \{i<k: a_{m}R_{i}a_{n}\}$.
By $\WPH^{h,2}_{k}$, take a weakly homogeneous set $H$ for color $i<k$ with $|H|>\min H$.
Write $H=\{f({h_{0}}), f({h_{1}}),\dots , f({h_{l-1}})\}$. Then, $l>f({h_{0}})$.
By definition, we have $a_{h_{n+1}}R_{i}a_{h_{n}}$ for any $n<n+1<l$.
Thus, we have
\[f_{i}(a_{h_{l-1}})<\dots<f_{i}(a_{h_{1}})<f_{i}(a_{h_{0}})< F_h(a_{h_{0}}) < f({h_{0}})<l,\]
which is a contradiction.

\vgt{$k\text{-}\HCT^h \Rightarrow k\text{-}\TT^h$}: Assume there exists a disjunctively well-founded transition invariant 
\[
	T= R_0 \cup \dots \cup R_{k-1} \supseteq R^+
\]
where each relation $R_i$ is well-founded and has height $\omega$ with weight function as in the hypothesis. By applying $k\text{-}\HCT^h$ their union is such that each transitive decreasing sequence $f'$ such that for any $n$ $f'(n+1)<F_h(f'(n))$ is finite. So it holds also for $R^+$, since it is preserved between subsets. Therefore, since $R$ is the graph of $f$ as in the hypothesis of $k\text{-}\TT^h$, there are no infinite $R$-decreasing sequences. Hence $R$ is well-founded.

\vgt{$k\text{-}\TT^h\Rightarrow\WPH^{h,2}_{k}$}: Assume by contradiction that there exist $X$ and $P: [X]^2 \to k$ such that, there is no weakly homogeneous $1$-large set. 
For any $i \in k$, define $R_i$ as follows:
\[
	x R_i y \iff y < x  \wedge P(\bp{y,x}) = i.
\] 
We claim that $R_i$ has height $\omega$ for any $i\in k$. In fact, if $X=\{x_i : i \in \N\}$, we can define a weight function $f_i: X \to \N$, by bounded recursion:
\begin{align*}
f_i(x_n)&= 
\begin{cases}
x_0 &\mbox{if } n=0;\\
\min \left\{ \{f_i(x_m)-1 : m < n\ \wedge \ P(\bp{x_m,x_n})=i\} \cup \{x_n\}\right\} &\mbox{otherwise};
\end{cases}\\
f_i(x_n)&\leq x_n.
\end{align*}
$f_i$ is a weight function, since if $xR_i y$ then $P(\bp{y,x}) = i$ and $y < x$ and so $f_i(x) < f_i(y)$.
Moreover for any $x \in X$ we have $f_i(x) \geq 0$. Otherwise there should exist $y_0> \dots> y_l$ such that 
\[
	y_0 R_i y_1 \dots R_i y_l = x,
\]
where $l > y_0$, due to the definition of $f_i$ and since $X\subseteq \N$. This is a contradiction since we assumed that there is no weakly homogeneous sets for $P$. Then each $R_i$ has height $\omega$. 

Therefore, by applying $k\text{-}\TT^h$, $R = \bigcup\{ R_i : i \in k\}$ should be well-founded, but this is a contradiction since $R=[X]^2$. 
\end{proof}

We can consider also the relativized versions of the statements in Definition \ref{def1}. Formally:

\begin{definition}
For given $k\in\N$ and for given $X \subseteq \N$, we define the following statements.
\begin{enumerate}
 \item $\PHs^{2}_{k}$: for any infinite set $Y$ and any coloring function $P:[Y]^{2}\to k$, there exists a homogeneous set for $P$ which is $1$-large.
 \item $\WPHs^{2}_{k}$: for any infinite set $Y$ and any coloring function $P:[Y]^{2}\to k$, there exists a weakly homogeneous set for $P$ which is $1$-large.
 \item $k$-$\HCTo$: if $R_{i}$ is a binary relation on $S$ of height $\omega$ for any $i<k$, then $R_{0}\cup\dots\cup R_{k-1}$ is $H$-well-founded.
 \item $k\text{-}\TTo$: any relation $R$ for which there exists a disjunctively well-founded transition invariant composed of $k$-many relations in of height $\omega$ is well-founded.
\end{enumerate}
\end{definition}

It is easy to prove in $\RCAs$ that, by using the previous definitions, we have:
\[
	\WPHs^{2}_{k} = k\text{-}\HCTo = k\text{-}\TTo .
\]
For any fixed $k$, each of these statements proves that for any infinite set there exists a $1$-large subset, which is just another form of $\Ii^0$ (e.g. see \cite{SimpsonKeita}). Thus they are all equivalent to $\RCAo$ over $\RCAs$.  Nonetheless the full versions are not provable over $\RCAo$ and our results imply that within $\RCAs$
\[
	\forall k\ \WPHs^{2}_{k} = \forall k\ k\text{-}\HCTo = \forall k\ k\text{-}\TTo.
\]

\subsection{Termination Theorem for bounded relations} \label{Section: bounds}

Here we consider the formulations obtained by using bounds instead of weight functions. For the application to real programs, finding a bound seems to be as difficult as finding a weight function. So, this relaxed notion would be nonsense. On the other hand, $H$-bound could be found more easily, e.g., if $R$ has no loop and $a\in S$ has only $m$-many predecessors, then putting $f(a)=m$ is good enough to obtain a $H$-bound (Definition \ref{def:bounds}).

Then we can define the version of $H$-closure and of the Termination Theorem with bounds and with $H$-bounds. 

\begin{defi}
For given $h,k\in\N$, we define the following statements.
\begin{enumerate}
 \item $k$-$\HCT^{h}_b$:  If, for any $i<k$, $R_{i}$ is a binary relation with bound $f_i$ such that $f_i(n)<F_h(n)$ for any $n$, then any transitive $R_{0}\cup\dots\cup R_{k-1}$-decreasing sequence $f:\N \to \N$ such that for any $n\in \N$ $f(n+1)< F_{h}(f(n))$ is finite.
 \item $k$-$\HCT^{h}_H$: If, for any $i<k$, $R_{i}$ is a binary relation with $H$-bound $f_i$ such that $f_i(n)<F_h(n)$ for any $n$, then any transitive $R_{0}\cup\dots\cup R_{k-1}$-decreasing sequence $f:\N \to \N$ such that for any $n\in \N$ $f(n+1)< F_{h}(f(n))$ is finite.
 \item $k$-$\TT^h_b$: Let $R$ be a deterministic binary relation, whose transition function $f:\N \to \N$ is such that for any $n\in \N$ $f(n+1)< F_{h}(f(n))$. If there exists a disjunctively well-founded transition invariant for $R$ composed of $k$-many relations with bounds $f_i$ such that $f_i(n)<F_h(n)$ for any $n$, $R$ is well-founded.
 \item $k$-$\TT^h_H$: Let $R$ be a deterministic binary relation, whose transition function $f:\N \to \N$ is such that for any $n\in \N$ $f(n+1)< F_{h}(f(n))$. If there exists a disjunctively well-founded transition invariant for $R$ composed of $k$-many relations with $H$-bounds $f_i$ such that $f_i(n)<F_h(n)$ for any $n$, $R$ is well-founded.
\end{enumerate}
\end{defi}

Then, we have the following.
\begin{thm}[$\RCAs + \Tot(F_h)$]\label{theorem:WPHbimpliesTTb} For any natural number $k$,
 \[\WPH_{k}^{h,2}= k\text{-}\HCT^h_b= k\text{-}\TT^h_b.\]
\end{thm}
\begin{proof}
\vgt{$\WPH_k^{h,2}\Rightarrow k\text{-}\HCT^h_b$}: The argument of Theorem \ref{thm-1} provides a decreasing $R_i$-sequence from $a_{h_0}$ of length greater than $f_i(a_{h_0})$. Hence the thesis.

\vgt{$\HCT^h_b\Rightarrow k\text{-}\TT^h_b$}: This proof is the same of the one in Theorem \ref{thm-1}.

\vgt{$k\text{-}\TT^h_b \Rightarrow \WPH_k^{h,2}$}: Thanks to Theorem \ref{thm-1} we have that $k\text{-}\TT^h \Rightarrow \WPH_k^{h,2}$. Moreover $k\text{-}\TT^h_b$ implies $k\text{-}\TT^h$ since if the relation $R$ has a weight function then it has also a bound. Therefore we are done. 
\end{proof}

This implies that $k\text{-}\HCT^h_b = k\text{-}\HCT^h$ and $k\text{-}\TT^h_b = k\text{-}\TT^h$. In the case of $H$-bounds, the bounded versions are stronger. In fact, they are equivalent to the Paris Harrington Theorem for $k$-many colors.

\begin{thm}[$\RCAs + \Tot(F_h)$]\label{theorem: TTPH} For any natural number $k$,
 \[\PH_k^{h,2}= k\text{-}\HCT^h_H= k\text{-}\TT^h_H.\]
\end{thm}
\begin{proof}
\vgt{$\PH^{h,2}_{k}\Rightarrow k\text{-}\HCT^h_H$}. As we did in the proof of Theorem \ref{thm-1}, let binary relations $\bp{ R_{i}\subseteq S^{2}: i<k }$ and $H$-bounds $\bp{ f_{i}: S\to \N\ : i<k }$ be given as in the hypothesis of $k\text{-}\HCT^h_H$.
For the sake of contradiction, assume $R = \bigcup\{R_{i} : i < k\}$ and let $\langle a_{j}: j\in\N \rangle$ be an infinite transitive sequence for $R$ such that $a_{i+1}< F_h(a_i)$.
Define $f:\N \to \N$ as
\[
	f(n)=\max\bp{f_{i}(a_{n+1})+(n-j+1): i<k, j \leq n}.
\]
Let $X=\{f(j): j\in\N\}$, and define $P:[X]^{2}\to k$ as $P(\bp{f(n),f(m)})=\min \{i<k : a_{m}R_{i}a_{n}\}$.
By $\PH_{k}^{h,2}$, take a homogeneous set $H$ for color $i<k$ with $|H|>\min H$.
Write $H=\{x_{h_{0}}<x_{h_{1}}<\dots <x_{h_{l-1}}\}$. Then, $l>x_{h_{0}}$.
By definition, we have $a_{h_{n}}R_{i}a_{h_{m}}$ for any $m<n<l$.
Thus there exists a transitive decreasing $R_i$-sequence from $a_{h_0}$ of length $l$, but $f_i(a_{h_0}) < x_{h_0} < l$ which is a contradiction.

\vgt{$\HCT^h_H \Rightarrow k\text{-}\TT^h_H$}. As in Theorem \ref{thm-1}.
	
\vgt{$k\text{-}\TT^h_H \Rightarrow \PH_{k}^{h,2}$.} Assume that $P:[X]^2 \to k$ has no homogeneous $1$-large set. Then we define, as we did in Theorem \ref{thm-1},
\[
	xR_i y \iff y < x  \wedge P(\bp{y,x}) = i.
\]
Then let $f$ be the identity function on $X$. We claim it is a $H$-bound for any $R_i$. In fact, let $\langle x_j : j \in l\rangle$ be a decreasing transitive $R_i$-sequence: by definition unfolding we have $x_0 < x_1 < \dots < x_{l-1}$ and $x_b R_i x_a$ for any $0\leq a < b <l$. If $x_0=f(x_0)<l$ then we obtain a homogeneous $1$-large set and this is a contradiction. Then $x_0 = f(x_0) \geq l$: this means that the identity function is a $H$-bound. So thanks to $k\text{-}\TT^h_H$, $[X]^2 = \bigcup\{R_i : i \in k \}$ should be well-founded, and this is a contradiction.
\end{proof}

Also in this case, we can consider also the relativized versions of the statements above.

\begin{definition}
For given $k\in\N$ and for given $X \subseteq \N$, we define the following statements.
\begin{enumerate}
 \item $k$-$\HCToh$: if $R_{i}$ is a binary relation with bound for any $i<k$, then $R_{0}\cup\dots\cup R_{k-1}$ is $H$-well-founded.
 \item $k\text{-}\TToh$: any relation $R$ for which there exists a disjunctively well-founded transition invariant composed of $k$-many relations with bounds is well-founded.
  \item $k$-$\HCToH$: if $R_{i}$ is a binary relation with $H$-bound for any $i<k$, then $R_{0}\cup\dots\cup R_{k-1}$ is $H$-well-founded.
  \item $k\text{-}\TToH$: any relation $R$ for which there exists a disjunctively well-founded transition invariant composed of $k$-many relations with $H$-bounds is well-founded.
\end{enumerate}
\end{definition}

Hence within $\RCAs$:
\[
	\forall k\ \WPHs^{2}_{k} = \forall k \ k\text{-}\HCToh = \forall k \ k\text{-}\TToh.
\]
\[
	\forall k \ \PHs^{2}_{k} = \forall k \ k\text{-}\HCToH = \forall k \ k\text{-}\TToH.
\]

\section{Bounding via Fast Growing Hierarchy} \label{Section: boundedbounds}

Is there a correspondence between the complexity of a primitive recursive transition relation and the number of relations which compose the transition invariant? Here we prove that actually there is, by applying the results obtained in Section \ref{Section: boundedtermination}. 
As a corollary we get the following proposition.

\begin{proposition}
For a transition relation  $R\subseteq \N^{2}$, the following are equivalent.
\begin{itemize}
 \item[(1)] $R$ is primitive recursively bounded.
 \item[(2)] $R$ is $k$-disjunctively linearly H-bounded for some $k\in\omega$.
\end{itemize}
Moreover, if $R$ is a deterministic transition relation, the following is also equivalent to the above.
\begin{itemize}
 \item[(3)] $R$ is $k$-disjunctively linearly bounded for some $k\in\omega$.
\end{itemize}
\end{proposition}

\subsection{From transition invariants to bounds} \label{subsection: FirstComparison}
The main result we are going to prove here is the equivalence between  $\forall k \ k\text{-}\TTo$ and the relativized version of the totality of the fast growing hierarchy.

Given an increasing function $f: \N \to \N$ and a natural number $k$, we define $F_{k,f}$ as follows: 
\begin{align*}
\begin{cases}
 F_{0,f} (x) = f(x)+1,\\
 F_{n+1,f}(x)=F_{n,f}^{(x+1)}(x).
\end{cases}
\end{align*}

Then $\Ack_k^f$ is the closure under limited recursion and substitution of the set of functions defined by constant, projections, sum and $F_{h,f}$ for any $h \leq k$. For any natural number $k$, let $\Tot^*(\Ack_k)$ be the relativized version of $\Tot(\Ack_k)$, namely: for any function $f$, any function in $\Ack_k^f$ is total.

Summing up the results presented in this subsection we have the following:
\[ 
\forall k \ \WPHs^2_k = \forall k \ k\text{-}\TTo=\forall k \ k\text{-}\HCTo =\forall k \ \Tot^*(\Ack_{k}) = \forall k \ \PHs^2_k. 
\]

Recall that, although for any natural number $k$ $\PHs^2_k$ and $\Tot^*(\Ack_k)$ hold within $\RCAo$, $\forall k \ \PHs^2_k$ and  $\forall k \ \Tot^*(\Ack_k)$ do not

\subsubsection{From termination to totality}
In order to prove that $\forall k \ k\text{-}\TTo \implies \forall k\ \Tot^*(\Ack_{k})$, we recall the following intermediate statement \cite{SolovayKetonen}.
We say that a finite set $X$ is $0$-\emph{large} if it is not empty, $X$ is $k+1$-\emph{large} if
\[
	X \setminus \bp{\min X} = \bigsqcup\bp{X_i : i \in n},
\]
where $n\geq \min X$ and $X_i$ are disjoint $k$-\emph{large} sets.

Notice that, for any set $X$ such that $X = \bigsqcup\bp{X_i : i \in n}$ for some $n > \min X$, if $X_i$ is $k$-large for all $i\in n$, then $X$ is $k+1$-large.

\begin{definition}
For given $h,k\in\N$ and for given $X \subseteq \N$, we define the following statements.
 \item $k$-$\LAR$ \footnote{Following the notation of \cite{Hajek} being $k$-large is equivalent to being $\omega^k$-large (or $\omega_0^k$-large). In this section we prove that in $\RCAs$ $k$-$\LAR^h$  is equivalent to $\Tot(\Ack_k^h)$ for any natural numbers $h,k$, as for $\omega^k$-large in \cite{Hajek}.}: any infinite set $X\subseteq \N$ contains a $k$-large set.
 \item $k$-$\LAR^h$: Given any  function $f: \N \to \N$ such that for any $n$ $f(n+1)< F_{h}(f(n))$, $\ran(f)$ contains a $k$-large set.
\end{definition}

\begin{proposition}[$\RCAo$]\label{proposition: WPHLAR}
For any $k\in\N$, we have
\[ \WPHs^{2}_{k}\Rightarrow k\text{-}\LAR.\]
\end{proposition}
\begin{proof}
Given any infinite $X \subseteq \N$, we want to find $L \subseteq X$ such that $L$ is $k$-large. For any $a,b \in \N$, let $[a,b) = \bp{x \in X : a \leq  x < b}$.  Define $k$-many sequences as follows:
\begin{align*}
	x^i_0 &= \min X;\\
	L_i &= \bp{y : \exists n \leq y (y= x^i_n)};\\
	x^i_{n+1} &= \min \left\{ y  : y \in L_i \wedge [x^i_n,y) \mbox{ is } i\mbox{-large}\right\}
\end{align*}

Observe that $x^0_n$ is the $n$-th element of $X$, and that $x^i_m$ may be undefined. However, given any $x \in X$, and $i,m \in \N$, we may decide whether $x$ is of the form $x^i_m$ or not. Any $x \in X$ is $x^0_n$ for some $n$. Moreover the following properties hold.

\begin{claim}\label{claim: WPHimpliesKlarge}
Let $i$ be a natural number.
\begin{enumerate}
\item If $a > \min X$ and $[a,b)$ is $i$-large then there exists $x^i_j \in (a,b]$ for some $j$.
\item If ${n_0} <\dots < {n_{l-1}}$ and $x^i_{n_0} < l$, then $[x^i_{n_0}, x^i_{n_{l-1}})$ is $(i+1)$-large.
\item If ${n_0} <\dots < {n_{l-1}}$ and $x^i_{n_0} < l$, then there exists $j \in \N$ such that $x^{i+i}_j \in (x^i_{n_0},x^i_{n_{l-1}}]$. 
\end{enumerate}
\end{claim}
\begin{proof}
 \begin{enumerate}
 \item  Let $j'$ be the maximum such that $x^i_{j'} \leq a$. Such element exists since $\min X = x^i_0 \leq a$. Then $[x^i_{j'}, b)$ is $i$-large, therefore $x^i_{j'+1}\leq b$ and by the choice of $j'$ $x^i_{j'+1}> a$. 
 \item  Observe that 
 \[
 		[x^i_{n_0}, x^i_{n_{l-1}})= [x^i_{n_0}, x^i_{n_{1}}) \cup \dots \cup [x^i_{n_{l-2}}, x^i_{n_{l-1}}).
 \]
 Since  $ n_{m+1} > n_m +1$ and $[x^i_{n_m}, x^i_{n_{m+1}})$ is $i$-large, we are done. 
 \item The thesis follows by points (1) and (2). \qedhere
%
%
 \end{enumerate}
\end{proof}

Then consider the coloring $P:[X]^2 \to k$, such that for any $x<y$ we have
\[
	P(\bp{x,y})=\max \left\{ i<k : \exists x^i_j \in [x,y) \right\}
\]
By applying $\WPHs^{2}_{k}$ there exists $H = \bp{h_0<\dots < h_{l-1}}$ which is weakly homogeneous and $1$-large. We claim that the color of this sequence is $k-1$, i.e. $P(\bp{h_0,h_1})=k-1$. There are two cases. If $h_0= \min X$ then $h_0=x^{k-1}_0$ and therefore $P(\bp{h_0,h_1})=k-1$ by definition. Otherwise assume by contradiction that $P(\bp{h_0,h_1})=i<k-1$ then for any $m < l-1$ the interval $[h_m, h_{m+1})$ contains an element of the form $x^i_{j_m}$.
Since $h_0 \neq \min X$, let $j_0$ be the maximum such that $x^i_{j_0}< h_0$. Therefore $x^i_{j_0} < x^i_{j_1} < \dots < x^i_{j_{l-2}}$ and, since $h_0 < l$, we have $x^i_{j_0} < l-1$. By applying point (3) we get that there exists $x^{i+1}_{j} \in (x^i_{j_0}, x^i_{j_{l-2}}]$, for some $j$. Since $x^{i+1}_{j} = x^i_{j'}$ for some $j'$ and by the definition of $j_0$ we have $x^{i+1}_{j} \geq h_0$.  Hence $x^{i+1}_j \in [h_0, h_{l-1})$ and this implies $P(\bp{h_0,h_{l-1}})\geq i+1$. Contradiction.

 By applying the same argument we have that $[h_0, h_{l-1})$ contains $(l-1)$-many $x^{k-1}_j$, hence $[h_0, h_{l-1})$ is $k$-large. 
\end{proof}

Due to the equivalence proved in the previous section we can easily obtain that for any natural number $k$:
\[ \forall k\ k\text{-}\TTo =\forall k\ k\text{-}\HCTo = \forall k \ \WPHs^2_k \geq \forall k\ k\text{-}\LAR.\]
Moreover we have that $\forall k \ k\text{-}\LAR \Rightarrow \forall k\  \Tot^*(\Ack_{k})$, and so the first goal is proved. 

\begin{proposition}[$\RCAo$] \label{proposition:LARtoTot}
$\forall k \ k\text{-}\LAR \Rightarrow \forall k \ \Tot^*(\Ack_{k})$
\end{proposition}
\begin{proof}
Given an increasing function $f:\N \to \N$, we define for any $k \in \N$ the function $f_k$ as follows:
\[
	f_k(a)=\min\left\{ |X| : X \mbox{ is $k$-large on } f''(a,\infty)\right\}+1.
\]
Observe that $f_0(a) > f(a)+1$. For any $k$, if $k\text{-}\LAR$ holds, then we can define $f_k$. We claim that for any $ a \in \N$ 
\[
	f_{k}(a)> f^{a+1}_{k-1}(a).
\]
Let $X$ be $k$-large on $f''[a, \infty)$. By definition of $k$-large there exists $n> a$ and there exists $X_i$ $(k-1)$-large for any $i\in n$ such that $|X \setminus \bp{\min X}|= \bigsqcup\bp{ |X_i| : i\in n}.$ Hence since $n \geq \min X \geq f(a)+1 \geq a+1$:
\[
	f_{k}(a)\geq |X| > f_{k-1}^n(a) \geq f^{a+1}_{k-1}(a).
\]

We claim that
\[
	 \forall k \forall a\  \exists b\ < f_k(a) (F_{k,f}(a)=b).
\]
We prove it by induction on $k$ (this sentence is $\Pi^0_1$).  If $k=0$, $F_{0,f}(a)=f(a)+1<f_0(a)$. Assume that it holds for $k-1$, then the it is true also for $k$ since $f_{k}(a) > f^{a+1}_{k-1}(a)$ and, by induction hypothesis, $f_{k-1}(a)$ is bigger than $F_{k-1,f}(a)$. This proves $\forall k \ \Tot^*(\Ack_k)$.
\end{proof}

Note that we can prove by external induction on $k$ that $\RCAs \vdash k\text{-}\LAR^h \implies \Tot(\Ack_k^h)$, by slightly modifying the argument of Proposition \ref{proposition:LARtoTot}.

\subsubsection{From totality to termination}

Here we apply a result by Solovay and Ketonen to prove that if we have $\forall k \ \Tot^*(\Ack_{k})$, then we have $\forall k \ k\text{-}\TTo$.

\begin{proposition}[$\RCAo$]\label{Lemma: AckimpliesLAR}
$\forall k \ \Tot^*(\Ack_{k}) \implies  \forall k \ k\text{-}\LAR$.
\end{proposition}
\begin{proof}
Let $X$ be a set and let $ f: \N \to X$ the increasing enumeration of $X$. We claim that
\[
	\forall k \forall n [n, F_{k,f}(n))\cap X \mbox{ is } k\text{-large}.
\]
We prove it by induction (since this formula is $\Pi^0_1$).  If $k=0$ then
\[
	f(n)\in [n,f(n)+1) \cap X.
\]
Suppose the statement is true for $k$ and let us prove it for $k+1$. Fix $n$, by definition 
$F_{k+1}(n)= F_{k,f}^{n+1}(n)$.  Hence we have
\[
	[n, F_{k+1,f}(n))\cap X =([n, F_{k,f}(n))\cap X) \cup \dots \cup  ([F_{k,f}^n(n), F_{k,f}^{n+1}(n))\cap X).
\]
Then it is $k+1 \text{-large}$, since for each $i \in n$, 
\[
	[F_{k,f}^i(n), F_{k,f}^{i+1}(n))\cap X
\]
is $k\text{-large}$.
\end{proof}

Observe that almost the same argument as above shows that for any natural number $k$, $\RCAs \vdash \Tot(\Ack_k^h) \implies k\text{-}\LAR^h$.

By Proposition  \ref{proposition:LARtoTot} and Proposition \ref{Lemma: AckimpliesLAR} we get:
\begin{corollary}[$\RCAs$]\label{Prop: AckimpliesLARGE} $\forall k \ \Tot^*(\Ack_{k}) = \forall k \ k\text{-}\LAR$.
\end{corollary}

In \cite{SolovayKetonen} Solovay and Ketonen proved the following result. 

\begin{theorem}[$\RCAs$, Solovay/Ketonen  \cite{SolovayKetonen}]\label{Theorem: LARimpliesWPH}
For any natural number $k$, 
\[(k+5)\text{-}\LAR^h \implies \PH^{h,2}_{k}\]
\end{theorem}

Thus, by composing Theorem \ref{Theorem: LARimpliesWPH} and Corollary \ref{Prop: AckimpliesLARGE} we obtain

\begin{corollary}[$\RCAo$] $\forall k \ \Tot^*(\Ack_{k}) \implies \forall k \ \PHs^2_k.$
\end{corollary}

Since for any $h$, $\PH_k^{h,2} \geq \WPH_k^{h,2} = k\text{-}\TT^h_b$, $\Tot(F_{k+h})\geq k\text{-}\LAR^h$ and thanks to Theorem \ref{Theorem: LARimpliesWPH} we get the following bound.

\begin{corollary} \label{cor: k+5}
For any $k, h \in \N$ and for any $R \subseteq \N^2$, $R$ is bounded by $F_{k+h+5}$ if there exist $R_0, \dots, R_{k-1} \subseteq \N^2$ such that $ R_0 \cup \dots \cup R_{k-1} \supseteq R^+$ and each $R_i$ is bounded by $F_h$.
\end{corollary}

Observe that the proofs of Proposition \ref{proposition: WPHLAR}, Proposition \ref{proposition:LARtoTot} and of Proposition \ref{Lemma: AckimpliesLAR} cannot be carried out within $\RCAs$. In fact $\Sigma^0_1$-induction is required to apply unbounded primitive recursive definitions, unbounded minimalization, and $\Pi^0_1$-induction. 

\subsection{From bounds to transition invariants} \label{Subection: linearlybounded}

Here we study a kind of vice versa of the results obtained in the previous subsections. Let $k$ be a natural number. Assume that we have a deterministic relation $R$ which is bounded by $F_k$, how many linearly bounded relations do we need to obtain a transition invariant? In here we prove that if $R$ is bounded by $F_k$ (the usual $k$-th fast growing function defined in the previous section) then it has a $k+2$-disjunctively well-founded transition invariant bounded by $F_0$ and that if $R$ is deterministic then there exists a $k+2$-disjunctively linearly bounded one. Up to now we do not know if such number is the minimum possible. 

\begin{thm}[$\RCAs + \Tot(F_k)$]\label{thm-sharp2}
For any deterministic  transition relation $R\subseteq\N^{2}$, $R$ is bounded by $F_{k}$ only if there exists $T_{0},\dots, T_{k+1}\subseteq \N^{2}$ such that $R^{+}\subseteq T_{0}\cup\dots\cup T_{k+1}$ and each $T_{i}$ is bounded by $F_{0}$.
\end{thm}
\begin{proof}
Let $R\subseteq \N^{2}$ be a deterministic  transition relation which is bounded by $F_{k}$. Note that for a deterministic  transition relation $R$ generated by a transition function, $R^{+}$ is $\Delta^{0}_{1}$-definable, thus it exists as a set within $\RCAs$.
Define $T_{<}$ and $T_{>}$ as
\begin{align*}
x T_{<} y &\leftrightarrow x R^{+}y \wedge x < y,\\
x T_{>} y &\leftrightarrow x R^{+}y \wedge x > y. 
\end{align*}
Trivially, $T_{<}$ is bounded by $F_{k}$ because every bound for $R$ is a bound for $R^+$ and for $T_<$, and $R^{+}=T_{<}\cup T_{>}$. 
Now, define $d_{T_{<}}:\N^{2}\to\N\cup\{\infty\}$ as
\begin{align*}
d_{T_{<}}(x,y)=
\begin{cases}
 \max\{m: \E \langle x_{i}: i\le m \rangle \mbox{ such that } x=x_{0}T_{<}x_{1}T_{<}\cdots T_{<}x_{m}=y \} & \mbox{if $xT_{<}y$},\\
 \infty & \mbox{otherwise}.
\end{cases}
\end{align*}
$T_>$ is decreasing, therefore it is bounded by $\id(x) = x$ and with more reason by $F_0(x) = x+1$. Hence we only need to decompose $T_{<}$ into $k+1$-many $F_{0}$-bounded relations.
For each $x$, by $F_k$-boundedness, we may effectively compute the list of states $x = x_0, \dots, x_m$ reachable from $x$ by $R$ and with more reason by $T_<$. In fact the finite sets of states which are reachable from $x$ is $\bigcup \bp{A_n : n < F_k(x)}$, where 
\[
A_n = \bp{y : \exists \ap{x_0, \dots, x_n} < y (x=x_0 R \dots R x_n=y)}.
\]
Thus, for each $i \le k$ we define by bounded induction in $\RCAs + \Tot(F_k)$ $\rank_{i}(x)$ for each state $x$ as follows:
\begin{itemize}
 \item for any $x\in \N$, $\rank_{i}(x)\ge 0$,
 \item $\rank_{i}(x)\ge n+1$ if there exists $y\in \bp{x_0, \dots, x_m}$ such that $d_{T_{<}}(x,y)\ge F_{i}(x)$ and $\rank_{i}(y)\ge n$,
 \item $\rank_{i}(x)=n$ if $\rank_{i}(x)\ge n$ and $\rank_{i}(x)\not\ge n+1$,
 \item $\rank_{i}(x) \leq m$.
\end{itemize}
Now, we put $T_{i}$ for $i\le k$ as follows:
\begin{align*}
 x T_{i} y\leftrightarrow x T_{<} y \wedge i=\min\{j\le k: \rank_{j}(x)=\rank_{j}(y)\}.
\end{align*}
Note that by definition, $T_{i}$ is transitive, and if $x T_{<} y$ then $\rank_{i}(x)\ge \rank_{i}(y)$ for any $i\le k$.
Since $R$ is deterministic, if $x T_{<} y$, $x T_{<} z$, and $\rank_{i}(y)>\rank_{i}(z)$, then $y T_{<} z$.

Now, for the sake of contradiction, we assume that $\langle x_{n}: n \le m \rangle$ is a $T_{i}$-sequence such that $m> F_{0}(x_{0})=x_{0}+1$.
Then, $\rank_{i}(x_{0})=\dots=\rank_{j}(x_{m})$.
If $i=0$, this is impossible since $d_{T_{<}}(x_{0},x_{m})\ge m>F_{0}(x_{0})$, which means $\rank_{0}(x_{0})>\rank_{0}(x_{m})$.
If $i>0$, then, $\rank_{i-1}(x_{0})>\rank_{i-1}(x_{1})>\dots>\rank_{i-1}(x_{m})$, and hence $\rank_{i-1}(x_{0})-x_{0}-1>\rank_{i-1}(x_{0})-m\ge \rank_{i-1}(x_{m})$.
By the definition of $\rank_{i-1}$, there exists $\langle y_{n}: n\le x_{0}+1 \rangle$ such that $x_{0}=y_{0}$, $d_{T_{<}}(y_{n},y_{n+1})\ge F_{i-1}(y_{n})$ and $\rank_{i-1}(y_{n})=\rank_{i-1}(x_{0})-n$. Recall that $x_0 = y_0 <y_1  <\dots <y_m$ since $T_{<}$ is an increasing relation. Moreover $d_{T_{<}}(y_{n},y_{n+1})\ge F_{i-1}(y_{n})$ implies that $y_{n+1} \ge F_{i-1}(y_n)$, since there exists a decreasing sequence composed of $F_{i-1}(y_n)$ from $y_{n+1}$.  Hence we have $d_{T_{<}}(y_{0},y_{x_{0}+1})\ge F_{i-1}(y_{x_{0}})\ge F_{i-1}(F_{i-1}(y_{x_{0}-1}))\ge\dots\ge F_{i-1}^{(x_{0}+1)}(y_{0})=F_{i-1}^{(x_{0}+1)}(x_{0})$. Since $\rank_{i-1}(y_{x_{0}+1})=\rank_{i-1}(x_{0})-x_{0}-1>\rank_{i-1}(x_{m})$, we have $y_{x_0 +1} \neq x_m$, therefore $y_{x_{0}+1} T_{<} x_{m}$.
Thus $d_{T_{<}}(x_{0},x_{m})>d_{T_{<}}(x_{0},y_{x_{0}+1})\ge F_{i-1}^{(x_{0}+1)}(x_{0})=F_{i}(x_{0})$.
This means $\rank_{i}(x_{0})>\rank_{i}(x_{m})$, which is a contradiction.
\end{proof}

The relations $T_i$ provided by the previous proof are not computable. We wondered how many relations would we need in order to have computable witnesses, but by now we do not know. 

Another question we may ask is: can we generalize this result for non-deterministic  transition relations? A partial answer to this question is the following: in the general case we need to use $H$-bounds instead of bounds.

\begin{thm}[$\RCAs + \Tot(F_k)$]\label{thm-sharp1}
For any transition relation $R\subseteq\N^{2}$, $R$ is bounded by $F_{k}$ only if there exists $T_{0},\dots, T_{k+1}\subseteq \N^{2}$ such that $R^{+}\subseteq T_{0}\cup\dots\cup T_{k+1}$ and each $T_{i}$ is H-bounded by $F_{0}$.
\end{thm}
\begin{proof}
Let $R\subseteq \N^{2}$ be a transition relation which is bounded by $F_{k}$.
Define $T_{<}$ and $T_{>}$ and $d_{T_{<}}$ as in the proof of Theorem \ref{thm-sharp2}.
Trivially, $T_{<}$ is H-bounded by $F_{k}$, $T_{>}$ is bounded by $F_{0}$, and $R^{+}=T_{<}\cup T_{>}$. Now we define $T_{i}$ for each $i\le k$ as
\begin{align*}
 x T_{i} y\leftrightarrow x T_{<} y \wedge i=\min \{j: d_{T_{<}}(x,y)\le F_{j}(x)\}.
\end{align*}
Then, we have $T_{0}\cup\dots\cup T_{k}=T_{<}$.
Thus, we only need to check that each $T_{i}$ is $F_{0}$-H-bounded.
Assume that $\langle x_{n}: n \le m \rangle$ is a $T_{i}$-homogeneous sequence such that $m> F_{0}(x_{0})=x_{0}+1$.
If $i=0$, this is impossible since $d_{T_{<}}(x_{0},x_{m})\ge m>F_{0}(x_{0})$.
If $i>0$, then, $d_{T_{<}}(x_{n},x_{n+1})>F_{i-1}(x_{n})$, and hence $x_{n+1}>F_{i-1}(x_{n})$ because any $T_i$-homogeneous decreasing sequence has decreasing values, and there is some $T_i$-homogeneous decreasing sequence from $x_{n+1}$ of length $F_{i-1}(x_{n})$.
Thus,
\[d_{T_{<}}(x_{0},x_{m})>d_{T_{<}}(x_{m-1},x_{m})>F_{i-1}(x_{m-1})>F_{i-1}(F_{i-1}(x_{m-2}))>\dots>F_{i-1}^{(m)}(x_{0})>F_{i}(x_{0}),\]
which is a contradiction because the definition of $T_i$-homogeneous sequence implies $x_0 T_i x_m$, therefore $d_{T_<}(x_0, x_m) \leq F_i(x_0)$.
\end{proof}

\section{Iterated version}\label{Section: iteratedversion}
Here we want consider iterated applications of termination theorems. For simplicity, we only consider $2$-disjunctive case with a linear $F_{0}$-bound.

For given $R\subseteq S^{2}$ and $S'\subseteq S$, $S'$ is said to be linearly $R$-connected if for any $x,y\in S'$, either $x R^{+} y$ or $y R^{+} x$. For instance, the set of elements of any $R$-sequence is $R$-linearly connected. Trivially, $R\subseteq S^{2}$ is well-founded if and only if for any $S'\subseteq S$ which is $R$-linearly connected, $R\cap S'^{2}$ is well-founded.

\begin{defi}
\begin{itemize}
 \item A binary relation $R\subseteq S^{2}$ is said to be $0$-depth linearly bounded if it is $F_{0}$-bounded.
 \item A binary relation $R\subseteq S^{2}$ is said to be $n+1$-depth linearly bounded if for any $S'\subseteq S$ which is $R$-linearly connected, there exist $T^{S'}_{1}, T^{S'}_{2}$ such that $T^{S'}_{1}\cup T^{S'}_{2}\supseteq R\cap S'^{2}$, and both $T^{S'}_{1}$ and $T^{S'}_{2}$ are $n$-depth linearly bounded.
\end{itemize}
\end{defi}

We can easily check that if $R$ is $k$-disjunctively linearly bounded then it is $k$-depth linearly bounded, but the converse does not hold. 

We define the notion $n$-depth linearly $H$-bounded similarly, but we replace $R$-linearly connected subset with $R$-homogeneous subset, where a set $S'\subseteq S$ is $R$-homogeneous if for any $x,y\in S'$, either $x R y$ or $y R x$. For instance any $R$-homogeneous sequence is a $R$-homogeneous set.

\begin{defi}
\begin{itemize}
 \item A binary relation $R\subseteq S^{2}$ is said to be $0$-depth linearly H-bounded if it is $F_{0}$-H-bounded.
 \item A binary relation $R\subseteq S^{2}$ is said to be $n+1$-depth linearly H-bounded if for any $S'\subseteq S$ which is $R$-homogeneous, there exist $T^{S'}_{1}, T^{S'}_{2}$ such that $T^{S'}_{1}\cup T^{S'}_{2}\supseteq R\cap S'^{2}$, and both $T^{S'}_{1}$ and $T^{S'}_{2}$ are $n$-depth linearly H-bounded.
\end{itemize}
\end{defi}

Again, we can easily check that if $R$ is $k$-disjunctively linearly bounded then it is $k$-depth linearly bounded, but the converse does not hold.

Now, applying the termination theorem $n$-times, we have the following.
\begin{prop}[$\RCAo$]\label{prop-dlb}
A transition relation $R\subseteq S^{2}$ is well-founded if it is $n$-depth linearly bounded.
\end{prop}
\begin{proof}
Assume that $R$ is not well-founded then $R^+$ is not well-founded, let $S'$ be the range of the infinite transitive $R^+$-sequence. By hypothesis there exists $T^{S'}_0$ and $T^{S'}_1$ such that they are $(n-1)$-depth linearly bounded and $T^{S'}_0 \cup T^{S'}_1 \supseteq R^+ \cap {S'}^2$. Then by the Termination Theorem there exists $i < 2$ such that $T^{S'}_1$ is not well-founded. By applying this argument $n$-many times we obtain $T^*$ such that it is $F_0$-bounded and it is not well-founded. This is a contradiction. 
\end{proof}

Similarly, by $n$-times applications of H-closure theorem, we have the following.
\begin{prop}[$\RCAo$]\label{prop-dlHb}
A transition relation $R\subseteq S^{2}$ is well-founded if it is $n$-depth linearly H-bounded.
\end{prop}
We want to calculate a bound for $R$ in the above propositions.
For this, we use several results from proof-theory and reverse mathematics.

In \cite{Bovykin} Bovykin and Andreas Weiermann introduced a notion of density which is weaker than the original one introduced by Paris \cite{Paris}: thanks to that they introduced the iterated version of $\PH^2_2$\footnote{$\PH^2_2$ id the lightface version of $\PHs^2_2$ for sets $X$ which are intervals.}. In order to extract $H$-bounds for the iterated version of Podelski and Rybalchenko's Termination Theorem, in Subsection \ref{subsection: It2} we deal with their definition. In  Subsection \ref{subsection: It1} we define a weak version which arises from $\WPH^2_2$\footnote{$\PH^2_2$ id the lightface version of $\WPHs^2_2$ for sets $X$ which are intervals.} and it is useful to extract bounds.

\subsection{$k$-depth linearly bounded}\label{subsection: It1}
For this case, a bound is given by the iterated version of $\WPH^{2}_{2}$.

\begin{defi}
For a finite set $X\subseteq\N$, we define the notion of \textit{$m$-w-density} (for $m\in\N$) as follows.
\begin{itemize}
  \item A finite set $X$ is said to be \textit{$0$-w-dense} if $|X|>\min X$.
  \item A finite set $X$ is said to be \textit{$m+1$-w-dense} if for any (coloring) function $P:[X]^2\to 2$, there exists a subset $Y\subseteq X$ such that $Y$ is $m$-w-dense and $Y$ is weakly $P$-homogeneous.
\end{itemize}
Note that ``$X$ is $m$-w-dense'' can be expressed by a $\Sigma_0^{0}$-formula.
Then, $m\WPH^{2}_{2}$ ($m$-th iterated $\WPH^{2}_{2}$) is the following assertion:
\begin{itemize}
 \item[] for any $a$ there exists a $m$-w-dense set $X$ such that $\min X>a$.
\end{itemize}
We put $WW_{m}:\N\to \N$ as $WW_{m}(x):=\min \{y: (x,y]$ is $m$-w-dense$\}$.
\end{defi}

\begin{thm}[$\RCAo$]
A transition relation $R\subseteq S^{2}$ is $WW_{k}$-bounded if it is $k$-depth linearly bounded.
\end{thm}
\begin{proof}
 Assume by contradiction that $R$ is $k$-depth linearly bounded and that there exists a $R$-sequence $X_k=\bp{a_n : n \in m_k}$ such that $m_k > WW_{k}(a_0)$. Put $Y_k = (a_0,m_k]$.
 Since $R$ is $k$-depth linearly bounded and $X_k$ is $R$-connected there exists $T_0$ and $T_1$ such that they are $(k-1)$-depth linear bounded and $T_0\cup T_ 1 \supseteq R^+ \cap X_k$. Define $P:[Y_k]^2 \to 2$ such that $P(\bp{n,m})= i$ for the minimum $i<2$ such that  $a_n T_i a_m$.
 Since $Y_k$ is $k$-w-dense, there exists $Y_{k-1}$ weakly $P$-homogeneous (then $T_i$-connected for some $i <2$) and $(k-1)$-w-dense. Define $X_{k-1} =\bp{a_n : n \in Y_{k-1}}$. By applying the same argument $k-1$-many times we obtain a relation $T^*$ $F_0$-bounded and $Y_0$ weakly $T^*$-homogeneous such that $Y_0$ is $0$-w-dense. Contradiction.
\end{proof}

Now, to know the increasing speed of $WW_{k}$, we use the following several results from proof-theory and reverse mathematics. Recall that the class provable recursive functions of $\WKLo+\CAC$ is exactly the same as the class of primitive recursive functions. On the other hand, we can see the following.

\begin{lem}\label{lemma: mthWPH}
Let $m\in \omega$.
Then, $m\WPH^{2}_{2}$ is provable from $\WKLo+\CAC$, in other words, $WW_{m}$ is a provably recursive function of $\WKLo+\CAC$.
\end{lem}
\begin{proof}
For any natural number $m$, $m\WPH^2_2$ can be proved by $m$ applications of $\WPH^2_2$. In fact, given a coloring $P: [\N]^2 \to 2$, define $X^m_i$ for any $i\leq m$ as follows. Put $X^m_{-1}= \bp{n : n > a}$. Assume we already defined $X^m_i$ and define $X^m_{i+1}$ to be the $1$-large, weakly homogeneous set provided by the application of $\WPH^2_2$ to the restriction of $P$ to $X^m_i$. It is quite easy to verify by induction that $X^{m}_0$ is the witness for $m\WPH^2_2$: it is $m$-w-dense set and $\min X^m_0>a$. In fact $X^0_0$ is $1$-large as required. Now consider a natural number $m$ and $X^{m+1}_0$. Observe that $\ap{X^{m+1}_i : i \leq m+1}$ is the sequence of witnesses needed to prove that $X^{m+1}_0$ is $m+1$-w-dense. 

Since $\CAC$ implies $\WPH^2_2$, the $m$-th iteration of $\WPH^2_2$ is provable from $\WKLo+\CAC$. 
\end{proof}

Combining these, we have a bound for Proposition~\ref{prop-dlb}.
\begin{thm}
A transition relation $R\subseteq S^{2}$ is primitive recursively bounded if it is $n$-depth linearly bounded for some $n\in\omega$.
\end{thm}

\subsection{$k$-depth linearly H-bounded}\label{subsection: It2}
This case is more complicated. Again, a $H$-bound is given by the iterated version of $\PH^{2}_{2}$. The next definition is the one introduced by Bovykin and Weiermann in \cite{Bovykin}. 

\begin{defi}
For a finite set $X\subseteq\N$, we define the notion of \textit{$m$-density} (for $m\in\N$) as follows.
\begin{itemize}
  \item A finite set $X$ is said to be \textit{$0$-dense} if $|X|>\min X$.
  \item A finite set $X$ is said to be \textit{$m+1$-dense} if for any (coloring) function $P:[X]^2\to 2$, there exists a subset $Y\subseteq X$ such that $Y$ is $m$-dense and $Y$ is $P$-homogeneous.
\end{itemize}
Note that ``$X$ is $m$-dense'' can be expressed by a $\Sigma_0^{0}$-formula. Then, $m\PH^{2}_{2}$ ($m$-th iterated $\PH^{2}_{2}$) is the following assertion:
\begin{itemize}
 \item[] for any $a$ there exists a $m$-dense set $X$ such that $\min X>a$.
\end{itemize}
We put $HH_{m}:\N\to \N$ as $HH_{m}(x):=\min \{y: (x,y]$ is $m$-dense$\}$.
\end{defi}
However, calculating the increasing speed of $HH_{m}$ is difficult. In order to study $H$-bounds we can use the following result:

\begin{theorem}[Bovykin/ Weiermann \cite{Bovykin}]
The set of $\Pi_2$-consequences of $\WKLo+\RT^{2}_{2}$ coincides with the set of $\Pi_2$-consequences of $\Ii + \bigcup\bp{m\PH^{2}_{2}: m\in \omega}$.
\end{theorem}

Thanks to the previous result we get the next corollary.
\begin{cor}
\begin{itemize}
\item $HH_{m}$ is a provably recursive function of $\WKLo+\RT^{2}_{2}$.
\item Every provably recursive function of $\WKL+\RT^{2}_{2}$ is bounded by $HH_{m}$ for some $m\in\omega$.
\end{itemize}
\end{cor}

In a recent work \cite{KeitaPatey}, Ludovic Patey and the second author proved the following:
\begin{thm}
$\WKLo+\RT^{2}_{2}$ is a ${\Pi}^0_3$-conservative extension of $\II$. 
\end{thm}

Thus, the provably recursive functions of $\WKLo+\RT^{2}_{2}$ are the primitive recursive ones, hence we can conclude that $HH_m$ is primitive recursive. 
Hence we have the following.
\begin{thm}
A  transition relation $R\subseteq S^{2}$ is bounded by $F_{m}$ for some $m \in \omega$ if it is $n$-depth linearly H-bounded for some $n\in\omega$.
\end{thm}

\section{Open Questions}

In Section \ref{Section: zoo} we proved that the full Termination Theorem is equivalent to the full Weak Ramsey Theorem. Moreover, since the full Weak Ramsey Theorem is provable within $\CAC$ plus full induction, which is strictly weaker than the full Ramsey Theorem for pairs, we answered negatively to \cite[Open Problem 3]{Gasarch}.
However we wonder if these two results have the same verification power, namely if the first order statements which are provable within $\forall k \WRT^2_k$ and within $\forall k \RT^2_k$ are the same. 
\begin{question}
Do $\forall k \WRT^2_k$ and $\forall k \RT^2_k$ have the same first order part? 
\end{question}

The proofs presented in Subsection \ref{subsection: FirstComparison} cannot be carried out within $\RCAs$. We conjecture that such results hold over $\RCAs$. Hence a question we wonder is:
\begin{question}
Does $\WPH^{h,2}_k$ imply $\Tot(\Ack_{\max\bp{k,h}})$ over $\RCAs$?
\end{question}

Eventually, in Subsection \ref{Subection: linearlybounded}, we proved that given a deterministic relation bounded by $F_k$ we can find a transition invariant composed of $(k+2)$-many linearly bounded relations. Moreover given a relation bounded in $F_k$ we provided a transition invariant composed of $(k+2)$-many linearly $H$-bounded relations. Thus we wonder whether such results are improvable:
\begin{question}
Given a relation bounded by $F_k$, what is the minimal number of linearly bounded relations whose union contains the transitive closure of $R$?
\end{question}

\paragraph{Acknowledgment}
The authors would like to thanks Stefano Berardi for his careful readings and his precious remarks. We are grateful to Emanuele Frittaion, Alexander Kreuzer, Alberto Marcone and Sylvain Schmitz  for their helpful suggestions.

\bibliographystyle{plain}
\bibliography{referencesThesis}

\end{document}